\documentclass[a4paper,12pt]{article}

\usepackage{amssymb}
\usepackage{epsfig}
\usepackage{amsfonts}
\usepackage{amsmath}
\usepackage{euscript}
\usepackage{amscd}
\usepackage{amsthm}
\usepackage{theoremref}
\usepackage{enumerate}
\usepackage{array}
\usepackage{mathrsfs}
\usepackage{tikz-cd}
\DeclareMathAlphabet{\mathpzc}{OT1}{pzc}{m}{it}
\DeclareMathAlphabet{\mathpzc}{OT1}{pzc}{m}{it}

\newtheorem{thm}{Theorem}[section]
\newtheorem{lem}[thm]{Lemma}
\newtheorem{prop}[thm]{Proposition} 
\newtheorem{cor}[thm]{Corollary}

\newtheorem{rem}[thm]{Remark}
\newtheorem{ex}[thm]{Example}

\newtheorem{Defn}[thm]{Definition}

\newcommand{\m}{\mathpzc{m}}
\newcommand{\p}{\mathpzc{p}}

\newcommand{\bZ}{\mathbb Z}
\newcommand{\Z}{\mathbb Z}

\newcommand{\bC}{\mathbb C}
\newcommand{\bN}{\mathbb \bZ_{>0}}
\newcommand{\A}{\mathbb A}

\newcommand{\x}{x_1,\dots,x_m}
\newcommand{\X}{X_1,\dots,X_m}
\newcommand{\mi}{1\leqslant i \leqslant m}

\newcommand{\td}{\operatorname{tr.deg}}

\newcommand{\Aut}{\operatorname{Aut}}

\newcommand{\gr}{\operatorname{gr}}
\newcommand{\dk}{\operatorname{DK}}
\newcommand{\ml}{\operatorname{ML}}

\title{ 
	On embedding of linear hypersurfaces
}
\author{
	Parnashree Ghosh$^*$, Neena Gupta$^{**}$ and Ananya Pal$^{***}$\\
	{\small{\it  Theoretical Statistics and Mathematics  Unit, Indian Statistical Institute,}}\\ 
	{\small{\it 203 B.T.Road, Kolkata-700108, India}}\\
	{\small{\it e-mail : $^*$parnashree$\_$r@isical.ac.in, ghoshparnashree@gmail.com}}\\
	{\small{\it e-mail : $^{**}$neenag@isical.ac.in, rnanina@gmail.com}}\\
	{\small {\it e-mail: $^{***}$palananya1995@gmail.com }}
}

\begin{document}
	\date{}
	\maketitle
	
	\abstract
	Linear hypersurfaces over a field $k$ have been playing a central role in the study of some of the challenging problems on affine spaces. Breakthroughs on such problems have occurred by examining two difficult questions on linear polynomials of the form\\ $H:=\alpha(X_1,\dots,X_m)Y - F(X_1,\dots, X_m,Z,T)\in D:=k[\X, Y,Z,T]$:\\
	(i) Whether $H$ defines a closed embedding of $\mathbb{A}^{m+2}$ into $\mathbb{A}^{m+3}$, i.e., whether the affine variety $\mathbb{V}\subseteq  \A^{m+3}_k$ defined by $H$ is isomorphic to $\A^{m+2}_k$.\\
	(ii) If $H$ defines a closed embedding $\mathbb{A}^{m+2}\hookrightarrow \mathbb{A}^{m+3}$ then whether $H$ is a coordinate in~$D$.
	
	\noindent
	Question~(i) connects to the Characterization Problem of identifying affine spaces among affine varieties; Question~(ii) is a special case of the formidable Embedding Problem for affine spaces.
	
	In \cite{adv2}, the first two authors had addressed these questions
	when $\alpha$ is a monomial of the form $\alpha(\X) = X_1^{r_1}\dots X_m^{r_m}$; $r_i>1,\, \mi$ and $F$ is of a certain type.  
	In this paper, using $K$-theory and $\mathbb{G}_a$-actions, we address these questions for a wider family of linear varieties.
	In particular, we obtain certain families of
	higher-dimensional hyperplanes $H$ (i.e., $H$ defines closed embedding) satisfying the Abhyankar–Sathaye Conjecture on the Embedding Problem. 
	For instance, we show that when the characteristic of $k$ is zero, $F \in k[Z,T]$ and $H$ defines a hyperplane, then $H$ 
	is a coordinate in 
	$D$ along with $X_1, X_2, \dots, X_m$.
	Our results in arbitrary characteristic yield counterexamples to the Zariski Cancellation Problem in positive characteristic.
	
	\smallskip
	
	\noindent
	{\small {{\bf 2020 MSC}. Primary: 14R10; 
			Secondary: 14R20, 13A02, 14R25, 13D15.		
	
	\section{Introduction}
	Throughout this paper, $k$ will denote a field of arbitrary characteristic (unless mentioned specifically) and $\overline{k}$ will denote an algebraic closure of $k$. All rings considered in this paper are commutative with unity. Capital letters like $X,Y,Z,T, U,V, X_1,\ldots,X_n$ etc., will denote indeterminates
	over the respective ground rings or fields. For a ring $R$,  we write $R^{[n]}$  to denote a polynomial ring in $n$ indeterminates over $R$.
	
	In \cite{kr}, H. Kraft listed eight fundamental problems on affine spaces, including the {\it Embedding Problem}, the {\it Zariski Cancellation Problem} (ZCP) and  
	the {\it Linearization Problem}.
	The  Embedding Problem asks the following:
	
	\smallskip
	\noindent
	{\bf Question 1}: Let $m , n$ be two positive integers with $n>m$. Is every closed embedding $\A_k^m \hookrightarrow \A_k^n$ rectifiable?  
	
	\medskip
	The ring theoretic formulation of Question~$1$ is the following question sometimes known as the {\it Epimorphism Problem}:
	
	\smallskip
	\noindent
	{\bf Question $\pmb{1^{\prime}}$}: Let $m,n$ be two positive integers  with $n>m$ and $\phi: k[X_1,\ldots,X_n] \twoheadrightarrow k[Y_1,\ldots,Y_m]$, a $k$-algebra epimorphism.
	Does it follow that there exists a system of coordinates $\{F_1,\ldots,F_n\}$ of $k[X_1,\ldots,X_n]$ such that $\ker \phi=(F_1,\ldots,F_{n-m})$?
	
	\medskip
	In particular, when $n-m=1$, i.e., when ker$(\phi)$ is a principal ideal $(H)$, we have the following version of the Embedding (or Epimorphism) Problem:
	
	\smallskip
	\noindent
	{\bf Question 2:} 	Let $k$ be a field. For some integer $n\geqslant 2$, let $H \in k[X_1,\ldots,X_n]$ be such that 
	$\frac{k[X_1,\ldots,X_n]}{(H)}=k^{[n-1]}$. Does it follow that $k[X_1,\ldots,X_n]=k[H]^{[n-1]}$? 	
	
	\medskip
	When $k$ is of positive characteristic, there are counterexamples to Question~$2$ given by B. Segre and M. Nagata (\cite{Se}, \cite{Na}).

	When $k$ is a field of characteristic zero and $n=2$,
	Abhyankar-Moh (\cite{AM}) and Suzuki (\cite{Suz}) gave an affirmative answer to Question~$2$ ---  this is popularly known as the celebrated ``Epimorphism Theorem''.
	The famous {\it Abhyankar-Sathaye Conjecture} asserts an affirmative answer to Question~$2$ when the characteristic of $k$ is zero.
	When $n>2$, partial affirmative results for Question~2 have been proved for specific forms of $H$, even when $k$ is of arbitrary characteristic (\cite{DG} gives a general survey).
	
	%
	
	For $n=3$, the first result was obtained for a ``linear plane" $H\in k[X_1,X_2,X_3]$, 
first by A. Sathaye (\cite{sp}) in characteristic zero and  later by P. Russell (\cite{rp}) in arbitrary characteristic.
	They showed that a coordinate system $\{X,Y,Z\}$ in $k^{[3]}$ can be chosen such that a linear plane in $k^{[3]}$
	takes the form $a(X)Y+b(X,Z)$; and $H$ becomes a coordinate along with $X$.	
	Other partial results for $n=3$ are given in \cite{Wright1}, \cite{rs} and \cite{DaDu1}.
	
	During the last few decades, some of the central problems in affine spaces crucially involved questions on certain linear polynomials:
	(i)	whether a specified linear polynomial $H\in k[X_1,\dots,X_n]$ is a hyperplane (i.e., whether $\frac{k[X_1,\dots,X_n]}{(H)}$ is a polynomial ring over $k$)
	and 
	(ii) whether linear hyperplanes of a certain form are coordinates -- a special case of Question~$2$ which seeks a possible generalisation of the  Sathaye-Russell Theorem.
	
	For example, a crucial step in settling the Linearization Problem for $\bC^{*}$-actions on $\bC^3$ involved deciding whether certain linear polynomials in $\bC^{[4]}$ defined by M. Koras and P. Russell,  like the Russell cubic 
	$$H(X,Y,Z,T)=X^2Y+X+Z^2+T^3,$$
	were hyperplanes. 
	The nontriviality of the Russell cubic $H$ above was first 
	shown by L. Makar-Limanov (\cite{ML1}) over a field of characteristic zero and later by A. Crachiola (\cite{cra}) over arbitrary fields.
	The full class of Koras-Russell threefolds were shown to be non-trivial by S. Kaliman and L. Makar-Limanov (\cite{KM}).
	
	Another powerful example of a linear affine variety is the ``Asanuma threefold" over a field $k$ of positive characteristic $p$, defined below: 
	$$
	R=\dfrac{k[X,Y,Z,T]}{(X^rY +Z^{p^e}+T+T^{sp})};\text{ where }r,e,s\geqslant 1, \, p^e\nmid sp \text{ and }sp\nmid p^e.
	$$
	It was first constructed by T. Asanuma as an example of a non-trivial $\mathbb{A}^2$-fibration (\thref{An_f}) over a PID not containing $\mathbb{Q}$ (cf. \cite{asa}).
	The example showed that the theorem of Sathaye (\cite{sp2}), establishing that any $\mathbb{A}^2$-fibration over a PID $S$, containing $\mathbb{Q}$, must be isomorphic to $S^{[2]}$, does not extend to the case $\mathbb{Q}\not\subseteq S$.
	Next,  Asanuma (\cite{asa3}) used the above example to construct a non-linearizable torus action on $\mathbb{A}^n_k$ over an infinite field $k$ of positive characteristic when $n\geqslant 4$. He also showed that $R^{[1]}=k^{[4]}$.
	
	The second author proved (\cite{inv}) that $R\not\cong_k k^{[3]}$, for $r\geqslant 2$, thus providing a negative solution to the ZCP (which asks whether $\mathbb{A}_k^n$ is cancellative) for the affine space $\mathbb{A}^3_k$ in positive characteristic.
	Next, following a question of Russell, she studied (\cite{com}) the more general family of linear threefolds
	$$
	R_1=\dfrac{k[X,Y,Z,T]}{(X^rY+F(X,Z,T))}, ~~r \geqslant 2
	$$
	over an arbitrary field $k$.
	Later in \cite{adv} and \cite{adv2}, the first two authors studied affine domains over an arbitrary field $k$ of the form
	$$
	R_m=\dfrac{k[\X,Y,Z,T]}{(X_1^{r_1}\dots X_m^{r_m}Y+f(Z,T)+X_1\dots X_mg(\X,Z,T))},\text{ for }r_i>1,\, \mi
	$$
	and established results connecting the Embedding Problem, the ZCP and the Dolgachev-Weisfeiler Affine Fibration Problem.
	
	In view of the wide applications of linear polynomials to central problems in Affine Algebraic Geometry, A.K. Dutta had asked the authors to consider a general  family of linear polynomials in $k[\X,Y,Z,T]$, over an arbitrary field $k$, of the following form:
	\begin{equation}{\label{H}}
		H:=	\alpha(\X)Y-f(Z,T)-h(\X,Z,T)
	\end{equation}
	such that $\alpha\notin k$ and $f\neq 0$ and investigate the following questions.
	
	\medskip
	\noindent
	{\bf Question 3}:
	(i) Under what condition $H$ defines a closed embedding $\mathbb{A}^{m+2}\hookrightarrow \mathbb{A}^{m+3}$ i.e., $A=k^{[m+2]}$?
	\begin{enumerate}
		\item[\rm(ii)] When $H$ defines a closed embedding, is it a part of a coordinate system of $k[\X,Y,Z,T]?$
		\item[\rm(iii)]	If so, is $H$ necessarily a coordinate along with $\X$ in $k[\X,Y,Z,T]$?
	\end{enumerate}
	
	\medskip 
	\noindent
	An affirmative answer to Question~3(ii) or Question~3(iii) would yield a higher-dimensional generalisation of the Sathaye-Russell Theorem on linear planes.
	
	%
		Let $A$ be an affine $k$-domain
		defined as follows:
		\begin{equation}\label{AH}
			A := \dfrac{k[\X,Y,Z,T]}{(H)}=\dfrac{k[\X,Y,Z,T]}{(\alpha(X_1,\dots,X_m)Y - f(Z,T) - h(X_1,\dots, X_m,Z,T))}. 
		\end{equation} 
		In most of our results, we shall also assume the following condition on $H$ (as in \eqref{H})
		\begin{equation}\label{Cond}
			\text{every prime factor of } \alpha \text{ divides } h \text{ in } k[\X,Z,T],
		\end{equation}
		a generalisation of the earlier conditions $\alpha=X_1^{r_1}\dots X_m^{r_m}$ and $X_1\dots X_m\mid h$.
		%
		%
		%
		%
		
		When $\alpha(\X) = X_1^{r_1}\dots X_m^{r_m}$, $r_i>1$ for all $ i\in \{1,\dots,m\}$ and $X_1\dots X_m\mid~h$, the first two authors have provided fourteen equivalent conditions for $A= k^{[m+2]}$ including the condition ``$f$ is a coordinate in $k[Z,T]$" and the condition ``$H$ is a coordinate in $k[\X,Y,Z,T]$" (cf. \cite[Theorem~3.10]{adv2} and \cite[Theorem~4.5]{GDA2}). 
		When $m=1$,
		S. Kaliman, S. V\'en\'ereau, M. Zaidenberg (\cite{kvz}) and S. Maubach (\cite{CDer})  have answered Question~3 over $\bC$ and M. El Kahoui, N. Essamaoui and M. Ouali  over a field of characteristic zero (\cite{interpolation}).

		In this paper we first study the ring ${A}$ defined in (\ref{AH}) and prove the following result (Theorems~\ref{stablyp}~and~\ref{thmA}).
		
		\medskip
		\noindent
		{\bf Theorem A.}
		Let $k$ be an algebraically closed field, $H$ be a polynomial as in \eqref{H} satisfying condition \eqref{Cond} and $A$ be as in \eqref{AH}.
		Suppose that $A^{[l]}=k^{[l+m+2]}$ for some $l\geqslant 0$ and ${k}[Z,T]/(f)$ is a regular domain.
		Then the following statements hold:
		\begin{enumerate}[\rm(i)]
			\item $\frac{{k}[Z,T]}{(f)}={k}^{[1]}$.
			\item If ch.$k=0,$ then $k[\X,Y,Z,T]=k[\X,H]^{[2]}$.
		\end{enumerate} 
		\medskip
		
		Using the above result we have proved the following theorem which establishes the Abhyankar-Sathaye Conjecture for a family of hyperplanes in $\A^{m+3}_k$ over a field $k$ of characteristic zero (\thref{mainch0}).
		Note that for a polynomial $p\in k[\X],$ $p_{X_i}$ denotes $\frac{\partial p}{\partial X_i}$.
		
		\medskip
		\noindent
		{\bf Theorem B.}
		Let $k$ be a field of characteristic zero and $A$ be an affine $k$-domain as in \eqref{AH} such that $H$ is a polynomial as in \eqref{H} satisfying condition \eqref{Cond}. 
		Let $\alpha= \prod_{i=1}^{n} p_i^{s_i}$ be a prime factorization of $\alpha$ in $k[X_1,\ldots,X_m]$. Suppose that one of the following conditions is satisfied:
		\begin{enumerate}[\rm (I)]
			\item $s_i=1$ for some $i$.
			
			\item $s_i>1$ for every $i$ and at least one of the following holds.
			\begin{itemize}
				\item [\rm (a)] $p_j^2 \mid h$ for some $j$.
				
				\item [\rm (b)] $(p_j, (p_j)_{X_1}, \ldots, (p_j)_{X_m}) k[X_1,\ldots,X_m]$ is a proper ideal for some $j$.
				
				\item [\rm (c)] $n \geqslant 2$ and $(p_l, p_j)k[X_1,\ldots,X_m]$ is a proper ideal for some $l\neq j$.
			\end{itemize}
			
		\end{enumerate} 
		Let $x_1,\ldots,x_m$ be the images in $A$ of $X_1,\ldots,X_m$ respectively.
		Then the following statements are equivalent:
		\begin{enumerate}[\rm(i)]
			
			\item  $k[\X,Y,Z,T]=k[\X,H]^{[2]}$.
			
			\item  $k[\X,Y,Z,T]=k[H]^{[m+2]}$.
			
			\item $A=k[\x]^{[2]}$.
			
			\item $A=k^{[m+2]}$.
			
			\item $k[Z,T]=k[f(Z,T)]^{[1]}$.
			
			\item  $A$ is an $\A^{2}$-fibration over $k[x_1,\ldots,x_m]$.

			\item  $A^{[l]}=k^{[m+l+2]}$ for some $l \geqslant 0$.

		\end{enumerate}
		
		\medskip
		Note that the hypersurfaces defined by 
		$$
		H= \alpha(\X)Y-f(Z,T) \in k[\X,Y,Z,T] \text{~with~} \alpha\notin k,\, f\neq 0
		$$
		are contained in the family of hypersurfaces mentioned in Theorem~B. 
		
		Next, we investigate Question~3 for hypersurfaces $H$ over a field $k$ of arbitrary characteristic, where $H$ is as in \eqref{H} and the $\alpha$ (in $H$) satisfies a certain property, defined in \thref{type A}, which we call ``${\bf r}$-divisible with respect to a certain coordinate system of $k^{[m]}$" for some ${\bf r}=(r_1,\dots,r_m)\in \mathbb{Z}^m_{>0}$.
		We first establish the following (Theorems~\ref{lin}~and~\ref{lin2}):
		
		\medskip
		\noindent
		{\bf Theorem C.}
		Let $k$ be an infinite field.
		Let $H$ be as in \eqref{H} such that $X_1\mid h$, $\alpha$ is a ${\bf r}$-divisible polynomial with respect to $\{\X\}$ in $k^{[m]}$, where ${\bf r}=(r_1,\dots,r_m)\in \mathbb{Z}_{>1}^{m}$ and $A$ be as in \eqref{AH}.
		Suppose  $\ml(A)= k$ or $\dk(A)=A$. Then there exist a system of coordinates $\{Z_1,T_1\}$ of $k[Z,T]$ and $a_0,a_1 \in k^{[1]}$, such that $f(Z,T)=a_0(Z_1)+a_1(Z_1)T_1$. 
		Furthermore, when $k[Z,T]/(f)=k^{[1]}$, then $k[Z,T]=k[f]^{[1]}$.
		\medskip
		
		Theorem~C will enable one to readily recognise the nontriviality of a large family of affine varieties (cf. \thref{notpoly}); for instance, the nontriviality of the varieties defined by the ``${\bf r}$-divisible polynomials" (cf. Examples \ref{ex1} and \ref{ex2})
		$$H=X^2(X+1)^2Y-(Z^2+T^3)-Xh_1(X,Z,T)\in k[X,Y,Z,T], \text{ for any } h_1\in k^{[3]}$$
		or 
		$$
		H=X_1X_2^2(X_1+X_2^2)Y -(Z^2+T^3)-X_2h_2(X_1,X_2,Z,T)\in k[X_1,X_2,Y,Z,T],\text{ for any }h_2\in k^{[4]}.
		$$
		
		Using Theorem~C, we have proved the following theorem (\thref{main}) which establishes equivalent conditions for $H$ to be a hyperplane and thus addresses Question~$3$ for the ${\bf r}$-divisible type hypersurfaces.
		
		\medskip
		\noindent
		{\bf Theorem D.}
		Let $A$ be an affine $k$-domain as in \eqref{AH} and $H$ be a polynomial as in \eqref{H} satisfying \eqref{Cond}.
		For ${\bf r}=(r_1,\dots,r_m)\in \Z_{>1}^m$, 
		let $\alpha$ be ${\bf r}$-divisible
		in the system of coordinates $\{X_1-\lambda_1,\dots,X_m-\lambda_m\}$, for some $\lambda_i\in \overline{k},\, \mi$  such that each $\lambda_i$ is separable over $k$.
		Let $x_1,\ldots,x_m$ be the images of $X_1,\ldots,X_m$ in $A$ respectively.
		Then the following statements are equivalent:
		\begin{enumerate}[\rm(i)]
			
			\item  $k[\X,Y,Z,T]=k[\X,H]^{[2]}$.
			
			\item  $k[\X,Y,Z,T]=k[H]^{[m+2]}$.
			
			\item $A=k[\x]^{[2]}$.
			
			\item $A=k^{[m+2]}$.
			
			\item $k[Z,T]=k[f(Z,T)]^{[1]}$.
			
		\end{enumerate}
		
		\medskip
		
		In fact, we have obtained equivalence of nine more statements involving stable isomorphisms, affine fibrations and two invariants -- the Makar-Limanov invariant and the Derksen invariant. 
		The family of hypersurfaces given by
		$$
		(X_1^{r_1+1} + X_1^{r_1}X_2^{r_2+1}+\dots +X_1^{r_1}\dots X_{m-1}^{r_{m-1}}X_m^{r_m+1}) Y -f(Z,T), \text{ for }r_i\geqslant 2,\, \mi
		$$
		and
		$$
		a_1(X_1) \cdots a_m(X_m) Y -f(Z,T)-  h(X_1,\ldots, X_m,Z,T),
		$$
		where every prime divisor of $a_1(X_1) \cdots a_m(X_m)$ in $k[X_1,\ldots,X_m]$ divides $h$, and every $a_i(X_i)$ has a separable multiple root $\lambda_i$ over $k$ are included in the family of hypersurfaces mentioned in Theorem~D.
		
		Question~3(i) is addressed by the
		equivalence $\rm(iv)\Leftrightarrow\rm(v)$ in 
		Theorem~B (in ch.$k=0$) and Theorem~D; 
		Question~3(ii) by the
		equivalence of $\rm(iv)\Leftrightarrow \rm(ii)$ and Question~3(iii) by $\rm(i)\Leftrightarrow\rm(ii)$.
		In particular, the  Abhyankar–Sathaye Conjecture holds affirmatively for the hypersurfaces $H$, considered in Theorems~B~and~D. 
		
		Theorem~D also yields a family of counterexamples to the Zariski Cancellation Problem in positive characteristic (\thref{czcp}).
		In a forthcoming paper \cite{zcp}, further discussions on their isomorphic classes have been undertaken.

		Section~\ref{preli} is on preliminaries; in Section~\ref{PropB} we study a few properties, like factoriality, of the affine domain $A$. Theorems~A~and~B will be proved in Section~\ref{THAB}; Theorems~C~and~D in Section~\ref{THC}.
		

		\section{Preliminaries}\label{preli}

		Throughout the paper, $k$ will denote a field 
		and for a ring $R$, $R^{[n]}$ denotes a polynomial algebra in $n$ indeterminates over $R$.  
		The notation $ \bZ,\bN$ stand for their usual meanings of all integers and all positive integers respectively. Throughout $m,\, n\in \bN$.
		
		For any field $L$, $\overline{L}$ will denote its algebraic closure and for any finite number of elements $\nu_1,\dots,\nu_n\in \overline{L}$, $L(\nu_1,\dots,\nu_n)$ denotes the smallest subfield of $\overline{L}$ containing $L\cup \{\nu_1,\dots,\nu_n\}$. 
		For a domain $R$, Frac($R$) denotes the field of fractions of $R$. 
		If $R\subseteq B$ are domains then $\td_R(B)$ denotes the transcendence degree of Frac($B$) over Frac($R$).
		For a ring $R$, a prime ideal $\p$ of $R$ and an $R$-algebra $B$, 
		$B_{\p}$ denotes the ring $S^{-1}B$, where $S:=R \setminus \p$,
		$k(\p)$ denotes the field $\frac{R_{\p}}{\p R_{\p}}\cong$ Frac$(\frac{R}{\p})$ 
		and $R^*$ denotes the group of all units of $R$.
		
		We now recall some known results.
		We first state the cancellative property of $k^{[1]}$ (\cite[(2.8)]{AEH}).
		\begin{thm}\thlabel{aeh}
			Let $B$ be $k$-domain such that $B^{[n]}= k^{[n+1]}$. Then $B = k^{[1]}$.
		\end{thm}
		
		Next we give a characterisation of $k^{[1]}$ for an algebraically closed field $k$ (cf. \cite[Lemma 2.9]{GFB}).
		\begin{lem}{\thlabel{alg}}
			Let $k$ be an algebraically closed field and $B$ a finitely generated $k$-algebra. Suppose that $B$ is a PID and $B^{*}= k^{*}$. Then $B=k^{[1]}$.
		\end{lem}
		%
		%
		
		We now state the well known Epimorphism Theorem due to Abhyankar-Moh and Suzuki (\cite{AM}, \cite{Suz}).
		
		\begin{thm}\thlabel{ams}
			Let $k$ be field of characteristic zero and $f \in k[Z,T]$. If  $\frac{k[Z,T]}{(f)}=k^{[1]}$, then $k[Z,T]=k[f]^{[1]}$.
		\end{thm}
		Segre and Nagata have constructed examples, ``non-trivial lines", to show that the above theorem does not hold over a field of positive characteristic (\cite{Se}, \cite{Na}).
		Let us recall the definitions of {\it line} and {\it non-trivial line} in this context.
		
		\begin{Defn}\thlabel{NonL}
			{\rm
				A polynomial $h \in k[X,Y]$ is said to be a {\it line} in $k[X,Y]$ if $\frac{k[X,Y]}{(h)} = k^{[1]}$. Furthermore, if $k[X,Y] \neq k[h]^{[1]}$, then $h$ is said to be a {\it non-trivial line} in $k[X,Y]$.}
		\end{Defn}
		
		Next we state a version of the Russell-Sathaye criterion for a ring to be a polynomial ring in one indeterminate over a given subring \cite[Theorem~2.3.1]{rs}, as presented in \cite[Theorem~2.6]{BD}.
		
		\begin{thm}\thlabel{rs}
			Let $C \subseteq D$ be integral domains such that $D$ is a finitely generated $C$-algebra. 
			Let $S$ be a multiplicatively closed subset of $C\setminus\{0\}$ generated by some prime elements of $C$ 
			which remain prime in $D$. 
			
			Suppose $S^{-1}D=(S^{-1}C)^{[1]}$ and, for every prime element $p \in S$, we have $pC=pD \cap C$ and $\frac{C}{(p)}$ is algebraically closed in 
			$\frac{D}{(p)}$. 
			
			Then $D=C^{[1]}$.
		\end{thm}
		
		Using this theorem, we now prove the following result.
		
		\begin{lem}\thlabel{RS}
			Let $R$ be an integral domain and $H:=\alpha Y -f(Z,T)-h(Z,T) \in R[Y,Z,T]$ be an irreducible polynomial,
			such that $\alpha\in R$ and it can be expressed as product of prime elements of $R$ and each of the prime factors of $\alpha$ divides $h$ in $R[Y,Z,T]$.
			Suppose $R[Z,T]=R[f]^{[1]}$. Then $R[Y,Z,T]=R[H]^{[2]}$.
		\end{lem}
		\begin{proof}
			Since $R[Z,T]=R[f]^{[1]}$, there exists $g \in R[Z,T]$ such that $R[Z,T]=R[f,g]$. Let $C= R[ H,g]$ and $D=R[Y,Z,T]$.
			Note that every prime divisor of $\alpha$ in $R$ remains prime in $C$ and $D$.
			Let	$S$ be the multiplicative closed subset of $C$ generated by the prime divisors of $\alpha$ in $R$. Then $S^{-1}D = (S^{-1}C)^{[1]}$ and $\frac{D}{qD} = \left(\frac{C}{qC}\right)^{[1]}$, for every prime divisor $q$ of $\alpha$ in $R$.	
			Therefore, by \thref{rs}, we have $D = C^{[1]}$.
		\end{proof}
		
		The next result on triviality of separable $\A^{1}$-forms over $k^{[1]}$ is a special case of a theorem of Dutta \cite[Theorem~7]{dutta}.
		
		\begin{lem}\thlabel{sepco}
			Let $f \in k[Z,T]$ be such that $L[Z,T]=L[f]^{[1]}$, for some separable field extension $L$ of $k$. Then $k[Z,T]=k[f]^{[1]}$. 
		\end{lem}
		We now recall a result proved by the second author from \cite[Proposition 3.6]{adv}. 
		
		\begin{prop}\thlabel{p1}
			Let $R$ be an integral domain, $\pi_1, \pi_2,\ldots, \pi_n \in R$ and $\pi=\pi_1 \pi_2\cdots \pi_n$. Let $G(Z,T) \in R[Z,T]$ be such that $R[Z,T]/(\pi, G(Z,T)) \cong_R (R/\pi)^{[1]}$. Let $r_1,\ldots, r_n$ be a set of positive integers and
			$D := R[Z, T,Y ]/(\pi_1^{r_1}\cdots \pi_n^{r_n}Y-G(Z,T)).$
			Then $D^{[1]}=R^{[3]}$.
		\end{prop}
		%
		%
				%
		%
		\section{Some properties of the ring $A$}\label{PropB}
		
		{\bf Throughout this section, $A$ will denote the following ring:}
		\begin{equation}\label{AA}
			A : = \dfrac{{k}[\X,Y,Z,T]}{(\alpha(X_1,\dots ,X_m)Y -f(Z,T)- h(X_1,\dots ,X_m,Z,T))}
		\end{equation}
		such that $\mathbf{f(Z,T)\neq 0}$ and {\bf every prime factor of \pmb{$\alpha$} in \pmb{$k[\X]$} divides \pmb{$h$} in \pmb{$k[\X,Z,T]$}}. Note that if $\alpha= 0$, then $h=0$ and hence $A$ is a polynomial ring if and only if $f(Z,T)$ is a line (cf. \thref{aeh}) and if ${\alpha} \in k^*$, 
		then $A=k^{[m+2]}$.	
		Henceforth, we assume that $\pmb{\alpha\notin k}$. 
		Note that $A$ and $A\otimes_k\overline{k}$ both are integral domains. 
		Let $x_1,\dots,x_m, y,z,t$ denote the images  of $X_1,\dots, X_m, Y,Z,T$ in $A$ respectively.
		{\bf {$\mathbf{E}$} will denote the subring of $\pmb A$ generated by $\mathbf{\x}$ over $\mathbf{k}$.}
		Note that $E~=~k[\x]=k^{[m]}$.
		
		\begin{lem}\thlabel{flatness 2}
			The ring $A$ is a flat $E$-algebra.
		\end{lem}
			%
			%
			%
			%
			%
			%
		\begin{proof}
			Follows from \cite[Corollary of Theorem~22.6]{matr}.
		\end{proof}
		
		We recall A. Sathaye's definition of an {\it $\A^n$-fibration} over a ring $R$ (\cite{sp2}).
		\begin{Defn}\thlabel{An_f}
			{\rm
				A finitely generated flat $R$-algebra $B$ over a ring $R$ is said to be an {\it $\A^n$-fibration} over $R$ if $B \otimes_R k(\p) = k(\p)^{[n]}$ for every prime ideal $\p$ of $R$.}
		\end{Defn}
		
		We prove an equivalent condition for $A$ to be an $\mathbb{A}^2$-fibration over $E$.
		\begin{lem}\thlabel{Fib2}
			The following statements are equivalent:
			\begin{enumerate}[\rm(i)]
				\item $A$ is an $\mathbb{A}^2$-fibration over $E$.
				\item $\dfrac{\frac{E_{\p}}{\p E_{\p}}[Z,T]}{(f(Z,T))}=\left(\dfrac{E_{\p}}{\p E_{\p}}\right)^{[1]}$, for all $\p\in Spec (E)$ with $\alpha(\x)\in \p$. 
			\end{enumerate}
		\end{lem}
		\begin{proof}
			$\rm(i)\Rightarrow\rm(ii):$ Since $A$ is an $\mathbb{A}^2$-fibration over $E$, we have 
			$$\frac{A_{\p}}{\p A_{\p}}=A\otimes_{E}\left(\dfrac{E_\p}{\p E_\p}\right) = \left(\dfrac{E_\p}{\p E_\p}\right)^{[2]}, \text{ for every } \p\in Spec(E).$$
			Let $\p\in Spec (E)$ be such that $\alpha\in \p$.
			Then 
			$$
			\left(\dfrac{E_\p}{\p E_\p}\right)^{[2]}= \frac{A_{\p}}{\p A_{\p}} = \left(\dfrac{\frac{E_{\p}}{\p E_{\p}}[Z,T]}{(f(Z,T))}\right)^{[1]}.
			$$
			Therefore, by 
			\thref{aeh}, we have $\dfrac{\frac{E_{\p}}{\p E_{\p}}[Z,T]}{(f(Z,T))}=\left(\dfrac{E_{\p}}{\p E_{\p}}\right)^{[1]}$.
			
			\noindent
			$\rm(ii)\Rightarrow\rm(i):$ By \thref{flatness 2}, $A$ is a flat $E$-algebra and
			therefore, it remains to show that 
			\begin{equation*}\label{2}
				A\otimes_{E}\left(\dfrac{E_\p}{\p E_\p}\right) = \frac{A_\p}{\p A_\p}=\left(\dfrac{E_\p}{\p E_\p}\right)^{[2]}, \text{ for every } \p\in Spec(E).
			\end{equation*}
			Let $\p\in Spec(E)$.  
			We now consider two cases:
			
			\noindent
			{\it Case 1}: $\alpha\notin \p$. Then $A_{\p}= E_{\p}[z,t] = E_{\p}^{[2]}$ and therefore, $\frac{A_{\p}}{\p A_{\p}}=\left(\dfrac{E_\p}{\p E_\p}\right)^{[2]}$.
			
			\noindent
			{\it Case 2}: $\alpha\in \p$.
			Then 
			$$
			\frac{A_{\p}}{\p A_{\p}}=\dfrac{\frac{E_{\p}}{\p E_{\p}}[Y,Z,T]}{(f(Z,T))}= \left(\dfrac{\frac{E_{\p}}{\p E_{\p}}[Z,T]}{(f(Z,T))}\right)^{[1]}= \left(\dfrac{E_\p}{\p E_\p}\right)^{[2]}.
			$$
		\end{proof}
		
		\begin{rem}\thlabel{fib1}
			{\rm 
				Note that if there exists a $\p\in Spec(E)$ such that $\alpha\in \p$ and $\frac{E_{\p}}{\p E_{\p}}$ is a separable field extension over $k$, then, by \thref{Fib2} and the fact that separable $\mathbb{A}^1$-forms over any field is trivial (cf. \cite[Lemma 5]{dutta}), it will follow that $A$ is an $\mathbb{A}^2$-fibration over $E$ if and only if $f(Z,T)$ is a line in $k[Z,T]$.}
		\end{rem}
		
		%
		%
		%
		
		We now prove a criterion for a simple birational extension of a UFD to be a UFD.
		
		\begin{prop}\thlabel{ufdg}
			Let $R$ be a UFD, $u,v\in R\setminus\{0\}$ and $C=\frac{R[Y]}{(uY-v)}$ be an integral domain.
			We consider $R$ as a subring of $C$.
			Let $u:= \prod_{i=1}^{n}u_i^{r_i}$ be a prime factorization of $u$ in $R$. Suppose that for every $i \in \{ 1, \dots, n\}$ for which $(u_i,v)R$ is a proper ideal, we have $\prod_{j\neq i}^{}u_j^{s_j}\notin (u_i,v)R,$ for arbitrary integers $s_j\geqslant 0$. Then the following statements are equivalent:  
			\begin{enumerate}[\rm(i)]
				\item $C$ is a UFD.
				
				\item For each $i$, $1 \leqslant i \leqslant n$, either $u_i$ is prime in $C$ or $u_i \in C^*$.
				
				\item For each $i$, $1 \leqslant i \leqslant n$, either $(u_i,v)R\in Spec(R)$ or $(u_i,v)R = R$, i.e., the image of $v$ in $\frac{R}{u_iR}$ is either a prime in $\frac{R}{u_iR}$ or a unit in $\frac{R}{u_iR}$. 
			\end{enumerate}  
		\end{prop}
		\begin{proof}
			
			$\rm (ii) \Leftrightarrow (iii):$
			For every $j, 1 \leqslant j \leqslant m$, we have
			\begin{equation}\label{xj}
				{\frac{C}{u_{j}C} \cong \left(\frac{R}{(u_j,v)}\right)^{[1]}.}
			\end{equation}
			Note that $u_j$ is either a prime element or a unit in $C$ according as $\frac{C}{u_{j}C}$ is either an integral domain or a zero ring. Hence the equivalence follows from \eqref{xj}.
			
			\smallskip
			\noindent
			$\rm (i) \Rightarrow (ii):$
			Note that $R \hookrightarrow C \hookrightarrow R[u_1^{-1},\ldots, u_n^{-1}]$. Suppose $u_j \notin C^{*}$ for some $j, 1 \leqslant j \leqslant m$. Since $C$ is a UFD, it is enough to show that $u_j$ is irreducible. Suppose $u_j= c_1c_2$ for some $c_1,c_2 \in C$. If $c_1,c_2 \in R$, then either $c_1 \in R^{*}$ or $c_2 \in R^{*}$, as $u_j$ is irreducible in $R$.
			Therefore, we can assume that at least one of them is not in $R$. Suppose $c_1 \notin R$. Let $c_1= \frac{h_1}{u_1^{i_1}\cdots u_n^{i_n}}$ and $c_2= \frac{h_2}{u_1^{l_1}\cdots u_n^{l_n}}$, for some $h_1,h_2 \in R$ and $i_s,l_s \geqslant 0$, $1 \leqslant s \leqslant n$. Therefore, we have 
			\begin{equation}\label{h}
				h_1h_2= u_j(u_1^{i_1+l_1} \cdots u_n^{i_n+l_n}).
			\end{equation}
			As $c_1 \notin R$, using \eqref{h}, without loss of generality, we can assume that  
			\begin{equation}\label{a11}
				c_1= \lambda \frac{\prod_{i\leqslant s}u_i^{p_i}}{\prod_{i= s+1}^{n}u_{i}^{p_{i}}},\text{for some } \lambda \in C^{*}\text{ and }  s < n,
			\end{equation}	
			where $ p_i \geqslant 0 \text{ for } 1\leqslant i\leqslant n, \text{~and~} p_i > 0 \text{~for~} i \geqslant s+1.$

			Now when $n = 1$ or $p_i=0$ for every $i \leqslant s$, then $c_1 \in C^{*}$, and we are done. If not, then  $n> 1$ and without loss of generality, we assume that $p_1>0$.
			
			Therefore, from \eqref{a11}, we have 
			$u_1^{p_1} \ldots u_s^{p_s} \in  u_iC \cap R= \left( u_i, v \right)R$ for every $i \geqslant s+1$.
			Hence by the given hypothesis, for every $i \geqslant s+1$, we get that $(u_i, v)R=R$, i.e., $u_i \in C^*$. 
			Thus we get 
			\begin{equation}\label{c}
				c_1=\mu \prod_{i\leqslant s} u_i^{p_i},
			\end{equation}
			for some $\mu \in C^*$ and hence
			$$
			\mu c_2= \frac{u_j}{\prod_{i\leqslant s} u_i^{p_i}}.
			$$
			If $\prod_{i\leqslant s} u_i^{p_i} \in u_jR$, then $c_2 \in C^*$ and we are done. If not, then
			$u_j \in u_iC \cap R=(u_i, v)R$ for every $i \leqslant s$ with $p_i>0$. 
			If for such an $u_i$ with $i \leqslant s$ and $p_i>0$, $(u_i, v)R$  is a proper ideal then we get a contradiction by the given hypothesis. Therefore, for all such $u_i$, $(u_i,v)R=R$ i.e., $u_i\in C^{*}$ and hence by \eqref{c}, $c_1 \in C^*$ and we are done.

			Therefore, we obtain that $u_j$ must be an irreducible element in $C$  and hence prime in $C$.  
			
			\smallskip
			\noindent
			$\rm (ii) \Rightarrow (i):$ 
			Without loss of generality we assume that $u_1,\ldots,u_{i-1} \in C^{*}$ and $u_{i},\ldots,u_n$ are primes in $C$
			for some $i$, $1\leqslant i \leqslant n$. 
			Since $C[u_{1}^{-1}, \ldots, u_{n}^{-1}]= C[u_{i}^{-1},\ldots,u_{n}^{-1}]=R[ u_{1}^{-1}, \ldots , u_{n}^{-1}]$ is a UFD, by Nagata's criterion for UFD (\cite[Theorem~20.2]{matr}), we obtain that $C$ is a UFD. 
		\end{proof}

		Next we deduce a necessary condition for the affine domain $A$ to be a UFD. 
		\begin{lem}\thlabel{UFDline}
			Suppose that $A$ is a UFD. Then either $f(Z,T)$ is irreducible in $k[Z,T]$ or $f(Z,T)\in k^{*}$.
		\end{lem}
		\begin{proof}			
			Let $\alpha(\X) = \prod_{i=1}^{n}p_i(\X)^{l_i}$ be a prime factorization of $\alpha$ in $k[\X]$. Putting $R=k[\X,Z,T]$, $u = \alpha(\X)$, $u_i=p_i,\, 1\leqslant i \leqslant n$ and $v=f(Z,T)+h(\X,Z,T)$ in \thref{ufdg}, we have $A=\frac{R[Y]}{(uY-v)}$. 
			Let 
			$$R_i:= \frac{R}{u_iR} = \frac{k[\X]}{(p_i)}[Z,T]$$
			and $x_{i1},\dots, x_{im},z_i,t_i$ denote the images of $\X,Z,T$ respectively in $R_i$, for $1~\leqslant~i~\leqslant~n$.
			Note that $vR_i = f(z_i,t_i)R_i$, for all $i,\,1\leqslant i\leqslant n$.  Therefore, for any $s_j\geqslant0$, $\prod_{j \neq i} p_j(x_{i1},\dots,x_{im})^{s_j} \notin vR_i$ whenever $vR_i$ is a proper ideal.
			Therefore, the assumption in \thref{ufdg} is satisfied and hence either $(p_1,f(Z,T))R$ is a prime ideal in $R$ or $(p_1,f(Z,T))R=R$. 
			Thus, either $f(z_1,t_1)$ is irreducible in $R_1$ or it is a unit in $R_1$. Hence it follows that either $f(Z,T)$ is irreducible in $k[Z,T]$ or $f(Z,T)\in k^{*}$.
		\end{proof}
		The next example shows that the converse of the above lemma is not true in general. 
		\begin{ex}
			{\rm
				Let 
				$$
				A=\dfrac{\mathbb{R}[X,Y,Z,T]}{((1+X^2)Y-(1+Z^2))}.
				$$
				Note that here $f(Z,T)=1+Z^2$ is irreducible in $\mathbb{R}[Z,T]$.
				Let $x$ be the image of $X$ in $A$.
				Further as
				$\frac{A}{(1+x^2)A} \cong \frac{\mathbb{C}[Y,Z,T]}{(1+Z^2)}$, it follows that $(1+x^2)$ is neither a unit nor a prime element in $A$. Thus, by \thref{ufdg}, it follows that $A$ is not a UFD. 
			}
		\end{ex}
		However, the converse of \thref{UFDline} is true if $f(Z,T)$ is irreducible in $\overline{k}[Z,T]$ or $f(Z,T)\in k^{*}$.
		\begin{lem}\thlabel{UFDC}
			Suppose that either $f(Z,T)$ is irreducible in $\overline{k}[Z,T]$ or $f(Z,T)\in k^{*}$. Then $A\otimes_k L$ is a UFD, for every algebraic extension $L$ of $k$.
		\end{lem}
		\begin{proof}
			Let $L$ be an algebraic extension of $k$, $A_L:= A\otimes_k L$ and $\alpha=\prod_{i=1}^{n}p_i^{l_i}$ be a prime factorization of $\alpha$ in $L[\X]$. 
			If $f(Z,T)\in k^*$, then 
			$$
			A_L=A_L[p_1^{-1}, \dots, p_n^{-1}]=L[\X,Z,T,p_1^{-1}, \dots, p_n^{-1}]$$ is a UFD.  
			Suppose that $f(Z,T)\notin k^*$  and that $f(Z,T)$ is irreducible in $\overline{k}[Z,T]$.  
			Since 
			$$
			A_L[p_1^{-1}, \dots, p_n^{-1}]=L[\X,Z,T,p_1^{-1}, \dots, p_n^{-1}]
			$$ 
			is a UFD, by Nagata's criterion for UFD (\cite[Theorem~20.2]{matr}), it is enough to show that 
			$p_1, \dots, p_n$ are primes in $A_L$. We show this below.
			
			Fix $i\in \{1,\dots,n\}$. Let $E_i= \frac{L[\X]}{(p_i)}$ and $F_i$ be the field of fractions of $E_i$. Since $f(Z,T)$ is irreducible in $\overline{k}[Z,T]$, it follows that $f$ is irreducible in $F_i[Z,T]$. Therefore, $\frac{F_i[Z,T]}{(f(Z,T))}$ is an integral domain and hence $\frac{E_i[Z,T]}{(f(Z,T))}$ is an integral domain.
			Thus
			$$
			\dfrac{A_L}{p_iA_L}=\dfrac{E_i[Y,Z,T]}{(f(Z,T))}=\left(\dfrac{E_i[Z,T]}{(f(Z,T))}\right)^{[1]}
			$$
			is an integral domain and hence $p_i$ is a prime in $A_L$. 
		\end{proof}

		We now determine some conditions on $\alpha$ and $h$  such that the regularity of $A$ implies the regularity of the ring $k[Z,T]/(f(Z,T))$.		
		
		\begin{lem}\thlabel{smooth2}
			Let $k$ be a perfect field and $R= k[X_1,\ldots,X_m,Y,Z,T]/(G)$ be an affine domain 
			where $G:= vY+h-f(Z,T)$, for some $v \in k[\X]$, $h \in k[\X,Z,T]$ and $f(Z,T)\in k[Z,T]$.				
			Let $v= \prod_{i=1}^{n} p_i^{s_i}$ be a prime factorization of $v$ in $k[X_1,\ldots,X_m]$. Suppose that one of the following conditions is satisfied:
			\begin{enumerate}[\rm (I)]
				\item $s_i=1$ for some $i$ and $p_i\mid h$.
				
				\item $s_i>1$ for every $i$ and at least one of the following holds.
				\begin{itemize}
					\item [\rm (a)] $p_j^2 \mid h$ for some $j$.
					
					
					\item [\rm (b)] $(p_j, (p_j)_{X_1}, \ldots, (p_j)_{X_m}) k[X_1,\ldots,X_m]$ is a proper ideal for some $j$ and $p_j\mid h$.
					
					\item [\rm (c)] 
					$(p_l, p_j)k[X_1,\ldots,X_m]$ is a proper ideal for some $l\neq j$ and $p_lp_j\mid h$.
				\end{itemize}
			\end{enumerate}
			Then $k[Z,T]/(f(Z,T))$ is a regular ring whenever $R$ is a regular ring. 
		\end{lem}
		
		\begin{proof}
			Let $D=k[\X,Y,Z,T]$ and	$I=(G, G_Y, G_{X_1}, \ldots, G_{X_m},  G_Z, G_T)D$.
			Since $k$ is a perfect field, $R$ is a regular ring if and only if $I=D$. Suppose that $R$ is a regular ring. 
			
			\smallskip
			\noindent
			(I) 	Without loss of generality, we assume that $s_1=1$ and hence $p_1\mid h$. 
			Let $\beta= \prod_{i \geqslant 2} p_i^{s_i}$ and $\tilde{h}={h}/{p_1}$. 
			Therefore, $G= vY+h-f=p_1(\beta Y+\tilde{h})-f$.
			Then it can be observed that $I \subseteq (f, f_Z, f_T, \beta Y+  \tilde h, p_1)D$. Since $R$ is regular, we have 
			\begin{equation}\label{f}
				(f, f_Z, f_T, \beta Y+  \tilde h, p_1)= D.
			\end{equation}
			Suppose, if possible, that $(f, f_Z, f_T) \subseteq \m k[Z,T]$ for some maximal ideal $\m$ of $k[Z,T]$. 
			Then as $p_1\nmid \beta$,
			$(\m, \beta Y+ \tilde h, p_1)$ is a proper ideal of $D$ containing $(f, f_Z, f_T, \beta Y+  \tilde h, p_1)D$, a contradiction. Thus $k[Z,T]/(f)$ must be a regular ring.

			\smallskip
			\noindent
			(II) Suppose condition (a) is satisfied. Then $I \subseteq (p_j, f, f_Z,f_T)D$
			and as $I=D$, we have $(f,f_Z,f_T)k[Z,T]=k[Z,T]$. 
			Now suppose conditions (b) or (c) is satisfied. When $p_j\mid h$, then $I \subseteq ( p_j,(p_j)_{X_1},\dots,(p_j)_{X_m}, f, f_Z, f_T)D$ and when $p_jp_l\mid h$ then $I\subseteq (p_l, p_{j}, f, f_Z, f_T)D$.  Since $v= \prod_{i=1}^{n} p_i^{s_i}\in k[\X]$ and as $I=D$, we have $(f,f_Z,f_T)k[Z,T]=k[Z,T]$ in either case. 
			
			Thus it follows that $k[Z,T]/(f)$ is a regular ring if any one of the conditions ($a$), ($b$) or ($c$) is satisfied.  
		\end{proof}
		However, the following result shows that when $k[Z,T]/(f)$ is a regular ring then $A$ is always regular.
		\begin{lem}\thlabel{smooth1}
			Let $k$ be a perfect field and $R= k[X_1,\ldots,X_m,Y,Z,T]/(G)$ an affine domain where $G= vY+h-f(Z,T)$, for some $v \in k[\X]$, $h \in k[\X,Z,T]$ and $f(Z,T)\in k[Z,T]$. Suppose that every prime factor of $v$ divides $h$ in $k[\X,Z,T]$. 
			Then $R$ is a regular ring if $k[Z,T]/(f(Z,T))$ is a regular ring.
		\end{lem}
		\begin{proof}
			Suppose, if possible, that $R$ is not a regular ring.
			Since $k$ is a perfect field, there exists a maximal ideal $M$ of $k[\X,Y, Z,T]$ such that $(G, G_Y, G_Z, G_T) \subseteq M$. Hence, there exists a prime factor $q$ of $v=G_Y$ such that $q \in M$.
			Now as $q\mid h,\, q\mid h_Z$ and $q\mid h_T$, we have   
			$$
			(f, f_Z, f_T) \subseteq M \cap k[Z,T]
			$$
			and hence $k[Z,T]/(f(Z,T))$ is not a regular ring. Therefore, our assumption is not possible and hence $R$ is  a regular ring.
		\end{proof}
		
		\begin{rem}
			{\em Note that the hypersurfaces defined by the polynomials
				$H= \alpha Y-~f(Z,T)$, where $\alpha \in k[\X]$,
				are contained in the family considered in \thref{smooth2}.}
		\end{rem}

		\section{On Theorems A and B}\label{THAB}

		We shall prove Theorems~A~and~B in Subsection~\ref{MAB}. 
		We first recall some relevant results from $K$-theory in Subsection~\ref{k-theory}. The reader can skip this subsection, if he/she is familiar with these results.
		
		\subsection{Some preliminary results on $K$-theory}\label{k-theory}
		
		\smallskip
		
		In this subsection we consider a Noetherian ring $R$ and 
		note some $K$-theoretic aspects of $R$ (cf. \cite{bass}, \cite{bgl}).
		Let $\mathscr{M}
		(R)$ denote the category of all finitely generated $R$-modules and
		$\mathscr{P}(R)$ denote the category of all finitely generated projective $R$-modules. Let $G_{0}(R)$ 
		and $G_{1}(R)$ respectively denote the {\it Grothendieck group} and the {\it Whitehead group} 
		of the category  $\mathscr{M}(R)$. Let $K_{0}(R)$ and $K_{1}(R)$ respectively denote the {\it Grothendieck group} 
		and the {\it Whitehead group} of the category  $\mathscr{P}(R)$.
		
		$G_0(R)$ is an Abelian group presented by generators $[M]$, where $[M]$ denotes the isomorphism class of objects of $\mathscr{M}(R)$, subject to the relations $[M] = [M^{\prime}] + [M^{\prime\prime}]$ for an exact sequence $0\rightarrow M^{\prime}\rightarrow M \rightarrow M^{\prime\prime}\rightarrow 0$ in $\mathscr{M}(R)$. 
		Let $\mathscr{M}^{\Z}(R)$ denote the category whose objects are pairs $(M,\sigma)$, where $M\in \mathscr{M}(R)$ and $\sigma$ is an $R$-linear automorphism of $M$. A morphism $\phi: (M,\sigma)\rightarrow (M^{\prime}, \sigma^{\prime})$ in $\mathscr{M}^{\Z}(R)$ is defined to be an $R$-linear map $\phi:M\rightarrow M^{\prime}$ such that $\sigma^{\prime}\phi = \phi \sigma$. $G_1(R)$ is generated by $[M,\sigma]$, where $[M,\sigma]$ denotes the isomorphism class of objects of $\mathscr{M}^{\Z}(R)$, subject to the relations $[M,\sigma] = [M^{\prime},\sigma^{\prime}] + [M^{\prime\prime},\sigma^{\prime\prime}]$ for an exact sequence $0\rightarrow (M^{\prime},\sigma^{\prime})\rightarrow (M,\sigma) \rightarrow (M^{\prime\prime},\sigma^{\prime\prime})\rightarrow~0$ in $\mathscr{M}^{\Z}(R)$ and $[M,\sigma\eta]= [M,\sigma]+[M,\eta]$ whenever $M\in \mathscr{M}(R)$ and $\sigma$ and $\eta$ are automorphisms of $M$. 
		%
		
		The next result from \cite[Appendix (4.10)]{bass} describe when the images of two elements from $\mathscr{M}(R)$ are equal in $G_0(R)$. 
		
		\begin{lem}\thlabel{K0}
			For $M,N\in \mathscr{M}(R)$, the following statements are equivalent.
			\begin{itemize}
				\item [\rm (i)] $[M]=[N]$ in $G_0(R)$.
				
				\item[\rm (ii)] There exist exact sequences $$0\rightarrow W^{\prime} \rightarrow X \rightarrow W^{\prime \prime} \rightarrow 0 \text{ and } 0\rightarrow W^{\prime} \rightarrow Y \rightarrow W^{\prime \prime} \rightarrow 0$$ in $\mathscr{M}(R)$ such that
				$M \oplus X \cong N \oplus Y$.
			\end{itemize}
		\end{lem}

		Adapting the proof of \thref{K0} and with the help of Whitehead's Lemma one can sketch a similar condition describing when two elements from $\mathscr{M}^{\bZ}(R)$ are equal in $G_1(R)$.
		
		\begin{lem}\thlabel{K1}
			For $(M,\sigma_M),(N,\sigma_N)\in \mathscr{M}^{\Z}(R)$, the following statements are equivalent.
			\begin{itemize}
				\item [\rm (i)] $[M, \sigma_M]=[N, \sigma_N]$ in $G_1(R)$.
				
				\item[\rm (ii)] There exist exact sequences 
				$$0\rightarrow (W^{\prime}, \sigma_{W^{\prime}}) \rightarrow (X,\sigma_X) \rightarrow (W^{\prime \prime}, \sigma_{W^{\prime \prime}}) \rightarrow 0$$ and 
				$$0\rightarrow (W^{\prime}, \epsilon\sigma_{W^{\prime}}) \rightarrow (Y,\sigma_Y) \rightarrow (W^{\prime \prime}, \sigma_{W^{\prime \prime}}) \rightarrow 0$$
				in $\mathscr{M}^{\Z}(R)$, where $\epsilon$ is composition of unipotent automorphisms on $W^{\prime}$, such that
				$(M \oplus X, \sigma_{M} \oplus \sigma_{X}) \cong (N \oplus Y, \sigma_N \oplus \sigma_Y)$.
			\end{itemize}
		\end{lem}
		
		For $i \geqslant 2$, the definitions of $G_{i}(R)$ and $K_{i}(R)$ 
		can be found in (\cite{sr}, Chapters 4 and 5).
		Let $C$ be a Noetherian ring and $\phi: R\rightarrow C$ be a flat ring homomorphism. Then for any $i\geqslant 0$, $G_i(\phi) : G_i(R)\rightarrow G_i(C)$ is
		induced by the functor $\otimes_R S: \mathscr{M}(R)\rightarrow \mathscr{M}(S)$, defined by $M\rightarrow M\otimes_R S$ (cf. \cite[5.8]{sr}).
		
		Next we recall a few theorems regarding  $G_i$'s. The following results can be found in \cite[Proposition~5.16 and Theorem~5.2]{sr}.
		%
		
		\begin{thm}\thlabel{fcom}
			Let $C$ be a Noetherian ring and $\phi: R \rightarrow C$ be a flat ring homomorphism. Let $t$ be a regular element of $R$ and $u:=\phi(t)$. Then the natural maps $\bar{\phi}: \frac{R}{tR}\rightarrow \frac{C}{uC}$ and $t^{-1}\phi : R[t^{-1}]\rightarrow C[u^{-1}]$ induce the following commutative diagram of long exact sequences of groups for $i\geqslant 1:$
			\begin{equation*}
				\begin{tikzcd}
					\cdots\arrow[r] &G_{i}(\frac{R}{tR}) \arrow[r]\arrow[d,"G_i(\overline{\phi})"] & G_{i}(R)\arrow[r]\arrow[d,"G_i(\phi)"] & G_{i}(R[t^{-1}])\arrow{r}\arrow[d,"G_i(t^{-1}\phi)"] & G_{i-1}(\frac{R}{tR}) \arrow{r}\arrow[d, "G_{i-1}(\overline{\phi})"] &\cdots  \\
					\cdots\arrow[r] &G_{i}(\frac{C}{uC}) \arrow[r] & G_{i}(C)\arrow[r] &G_{i}(C[u^{-1}]) \arrow{r}  & G_{i-1}(\frac{C}{uC}) \arrow{r} &\cdots.
				\end{tikzcd}
			\end{equation*}
		\end{thm}
		
		\begin{thm}\thlabel{split}
			For  an indeterminate $T$ over $R$, 
			%
				the map $G_{i}(R) \rightarrow G_{i}(R[T])$, induced by the inclusion $R \hookrightarrow R[T]$, is an isomorphism, for all $i \geqslant 0$.
				Hence $G_i(k[\X])~= ~G_i(k)$, for all $i\geqslant 0$.
				
		\end{thm}
		Now we recall some well-known facts about $G_0(R)$ and $G_1(R).$
		
		\begin{rem}\thlabel{rmk1}
			\rm{ 
				\begin{enumerate}[\rm(i)]
					\item 
					For a regular ring $R$, $G_{i}(R)=K_{i}(R)$, for every $i \geqslant 0$. 
					In particular, 
					\begin{enumerate}[\rm(a)]
						\item $G_0(k[\X])= G_0(k )= K_0(k)= \Z$
						\item $G_{1}(k[\X])=G_1(k)= K_1(k)=k^{*}$
					\end{enumerate}
					
					\item
					There is a canonical group homomorphism $\theta: R^{*}\hookrightarrow G_1(R)$ defined by $\theta(u)=[R,\sigma_{u}]$, for $u\in R^{*}$, where $\sigma_{u}:R\rightarrow R$ is the $R$-linear automorphism defined by $\sigma_{u}(r)=ru$ for $r\in R$.
					%
				\end{enumerate}
			}
		\end{rem}
		\subsection{Main Theorems}\label{MAB}
		
		\smallskip
		
		We first prove a main step towards Theorem~A. The authors thank Professor S. M. Bhatwadekar for fruitful discussions which led to the proof of the following lemma.

		\begin{lem}\thlabel{lem4}
			Let $k$ be an algebraically closed field. 
			Suppose that $C$ is a regular affine $k$-domain, $R$ is a reduced affine $k$-algebra and the map $R \hookrightarrow R\otimes_k C$ induces surjective maps $G_i(R) \rightarrow G_i(R\otimes_k C)$ for $i=0,1$. Then the canonical inclusion $\tau: k \hookrightarrow C$ induces isomorphisms of $K_i$-groups for $i=0,1$ and hence $K_0(C)=\bZ$ and $K_1(C)=k^*$.	
		\end{lem}
		\begin{proof}
			Observe that the inclusion $\tau : k\hookrightarrow C$ induces inclusions $\iota: K_0(k)=\bZ \hookrightarrow K_0\left( C \right)$ and $\eta: K_1(k)=k^* \hookrightarrow ~K_1(C)$. 
			We first show that $\iota$ is a surjective map.
			
			Let $P_C$ be a finitely generated projective $C$-module. 
			Since $G_0(R) \rightarrow G_0(R \otimes_k C)$ is surjective, there exist finitely generated $R$-modules $M_R$ and $N_R$ such that 
			$$[R \otimes_k P_C]=[M_R \otimes_k C]-[N_R \otimes_k C].$$ Therefore, by \thref{K0}, there exist exact sequences of $R \otimes_k C$-modules 
			$$0\rightarrow W^{\prime} \rightarrow T_j \rightarrow W^{\prime \prime} \rightarrow 0,\text{ for }j=1,2,$$ such that
			\begin{equation}\label{c2}
				(R \otimes_k P_C) \oplus (N_R \otimes_k C)\oplus T_1 \cong (M_R \otimes_k C) \oplus T_2.
			\end{equation} 
			Let $S$ be the set of all non-zero divisors of $ R$. Since $R$ is a reduced affine $k$-algebra, $S^{-1}R= \prod_{i=1}^{n} L_i$, where $L_i$ is a finitely generated field extension of $k$ for each $i,\, 1\leqslant i\leqslant n$.
			Therefore, from \eqref{c2}, we have
			\begin{equation}\label{c3}
				( S^{-1}R \otimes_k P_C)\oplus(S^{-1}N_R \otimes_k C)  \oplus S^{-1}T_1 \cong (S^{-1}M_R \otimes_k C) \oplus S^{-1}T_2.
			\end{equation}
			Then from \eqref{c3}, we have,	
			\begin{equation*}\label{j6}
				(L_1\otimes_k P_C) \oplus (L_1 \otimes_R N_R \otimes_k C) \oplus (L_1 \otimes_R T_1) \cong (L_1 \otimes_R M_R \otimes_k C) \oplus ( L_1 \otimes_R T_2).
			\end{equation*}
			Further it follows that for $j=1,2$,
			$$
			0\rightarrow L_1 \otimes_R W^{\prime} \rightarrow L_1 \otimes_R T_j \rightarrow L_1 \otimes_R W^{\prime \prime} \rightarrow 0
			$$ 
			are exact sequences of $\widetilde C:= L_1\otimes_k C$ modules.
			Note that $L_1\otimes_{R} M_R \cong L_1^{r_1}$ and $L_1\otimes_{R} N_R \cong L_1^{r_2}$ as $L_1$-vector spaces, for some $r_1,r_2>0$.
			Now, by \thref{K0}, it follows that 
			$$[L_1 \otimes_k P_C]=[L_1\otimes_{R} M_R\otimes_k C]-[L_1\otimes_{R} N_R\otimes_k C] = [\widetilde{C}^{r_1}]-[\widetilde{C}^{r_2}]\text{ in }G_0(\widetilde{C}).$$
			Note that $\widetilde{C}$ is a regular ring as $k$ is algebraically closed. Therefore, $ G_0(\widetilde{C})=K_0(\widetilde{C})$.
			Hence, there exists integer $s\geqslant0$ such that 
			\begin{equation}\label{c6}
				(L_1 \otimes_k P_C)\oplus\widetilde{C}^{r_2+s}\cong \widetilde{C}^{r_1+s}.
			\end{equation}
			Since $P_C$ is a finitely generated projective $C$-module, there exists a finitely generated $k$-algebra $\widehat{L}$ such that 
			\begin{equation}\label{c7}
				(\widehat{L} \otimes_k P_C)\oplus(\widehat{L} \otimes_k C)^{r_2+s}  \cong ( \widehat{L} \otimes_k C)^{r_1 +s}.
			\end{equation}
			Now let $\m$ be a maximal ideal of $\widehat{L}$. 
			Then $\frac{\widehat{L}}{\m}=k$, as $k$ is an algebraically closed field. 
			Therefore, as ${(\widehat{L}\otimes_k C)}/{\m(\widehat{L}\otimes_k C)}\cong C$ and $({\widehat{L}\otimes_k P_C})/{\m(\widehat{L}\otimes_k P_C)}\cong P_C$, from \eqref{c7}, we have $$[P_C]=[C^{r_1}]-[C^{r_2}] \text{ in } K_0(C).$$ 
			Therefore, the map $\iota$ is surjective, and thus $K_0(C) =\mathbb{Z}$.
			%
			%
			
			Similarly, we show that $\eta: K_1(k)=k^* \hookrightarrow K_1(C)$ is a surjective map.
			We consider an element $[Q_C, \sigma_C] \in K_1(C)$.
			Since $G_1(R) \rightarrow G_1(R \otimes_k C)$ is surjective, we have finitely generated $R$-modules $\widetilde{M}_R$ and $\widetilde{N}_R$ with $\gamma_R \in \Aut(\widetilde{M}_R)$ and $\delta_R\in \Aut(\widetilde{N}_R)$ such that
			$$
			[R \otimes_k Q_C,  1_R \otimes \sigma_C]=[\widetilde M_R \otimes_k C, \gamma_R \otimes_k 1_C]-[\widetilde N_R \otimes_k C, \delta_R \otimes_k 1_C]
			$$  
			in $G_1(R \otimes_k C)$. Therefore, by \thref{K1}, there exist
			exact sequences in the category $\mathscr{M}^{\Z}(R \otimes_k C)$ (defined as in Subsection~\ref{k-theory})
			$$
			0\rightarrow (U^{\prime},\sigma_{U^{\prime}} )\rightarrow (T^{\prime}_1,\sigma_{T^{\prime}_1}) \rightarrow (U^{\prime \prime},\sigma_{U^{\prime\prime}}) \rightarrow 0 
			$$
			and 
			$$
			0\rightarrow (U^{\prime},\epsilon\sigma_{U^{\prime}} )\rightarrow (T^{\prime}_2,\sigma_{T^{\prime}_2}) \rightarrow (U^{\prime \prime},\sigma_{U^{\prime\prime}}) \rightarrow 0,
			$$
			where $\epsilon$ is composition of unipotent automorphisms on $U^{\prime}$ such that
			\begin{equation}\label{c4}
				(R \otimes_k Q_C, 1_R \otimes \sigma_C)\oplus(\widetilde N_R \otimes_k C, \delta_R \otimes 1_C)  \oplus (T_1^{\prime}, \sigma_{T_1^{\prime}}) \cong (\widetilde M_R \otimes_k C, \gamma_R \otimes 1_C) \oplus (T_2^{\prime}, \sigma_{T_2^{\prime}}).\\
			\end{equation}
			Now, localising equation \eqref{c4} by $S$, and considering the first component, we have 
			the following
			\begin{equation*}\label{c5}
					(L_1 \otimes_k Q_C ,1_{L_1} \otimes \sigma_C)\oplus
					(\widetilde{C}^{s_2}, \delta_{L_1} \otimes 1_C) \oplus (L_1 \otimes_R T_1^{\prime},\sigma_{(L_1 \otimes_R T_1^{\prime})}) 
					\cong
				\end{equation*}
				\begin{equation*}
					(\widetilde{C}^{s_1}, \gamma_{L_1} \otimes 1_C) \oplus (L_1 \otimes_R T_2^{\prime},\sigma_{(L_1 \otimes_R T_2^{\prime})}),
			\end{equation*}
			where $L_1\otimes_R \widetilde M_R\cong L_1^{s_1}$ and $L_1\otimes_R \widetilde N_R\cong L_1^{s_2}$ as $L_1$-vector spaces, for some integers $s_1, s_2>0$ with $\gamma_{L_1}\in ~\Aut(L_1^{s_1})$ and $\delta_{L_1} \in \Aut(L_1^{s_2})$
			along with the following exact sequences in the category $\mathscr{M}^{\Z}(\widetilde{C})$
			$$
			0\rightarrow (L_1 \otimes_R U^{\prime},\sigma_{L_1\otimes_R U^{\prime}}) \rightarrow( L_1 \otimes_R T^{\prime}_1,\sigma_{L_1\otimes_R T^{\prime}_1}) \rightarrow (L_1 \otimes_R U^{\prime \prime},\sigma_{L_1\otimes_R U^{\prime\prime}}) \rightarrow 0,
			$$ 
			and
			$$
			0\rightarrow (L_1 \otimes_R U^{\prime},(\epsilon \sigma)_{L_1\otimes_R U^{\prime}}) \rightarrow( L_1 \otimes_R T^{\prime}_2,\sigma_{L_1\otimes_R T^{\prime}_2}) \rightarrow (L_1 \otimes_R U^{\prime \prime},\sigma_{L_1\otimes_R U^{\prime\prime}}) \rightarrow 0.
			$$
			Now using \thref{K1} we have, 
			$$[ L_1 \otimes_k Q_C ,  1_{L_1} \otimes \sigma_C ]+[\widetilde{C}^{s_2},\delta_{L_1} \otimes 1_C]=[\widetilde{C}^{s_1},   \gamma_{L_1} \otimes 1_C   ] \text{ in } G_1(\widetilde{C}).$$ 
			As before $G_1(\widetilde{C})=K_1(\widetilde{C})$, and
			thus there exists an integer $r>0$ and an automorphism $\tilde{\sigma}_{\widetilde{C}^r}  \in \Aut (\widetilde{C}^r)$ such that 
			\begin{equation}\label{e6}
				(L_1 \otimes_k Q_C,  1_{L_1}  \otimes \sigma_C) \oplus (\widetilde{C}^{s_2}, \delta_{ L_1} \otimes 1_C) \oplus(\widetilde{C}^r, \tilde{\sigma}_{\widetilde{C}^r}) \cong (\widetilde{C}^{s_1}, \gamma_{ L_1} \otimes 1_C) \oplus (\widetilde{C}^r, \tilde{\sigma}_{\widetilde{C}^r}).
			\end{equation}
			Since $Q_C$ is a finitely generated projective $C$-module, there exists a finitely generated $k$-algebra $\widetilde{L}$ such that  
			\begin{equation*}\label{c18}
				(\widetilde{L} \otimes_k Q_C,  1_{\widetilde{L}}  \otimes \sigma_C) \oplus((\widetilde{L} \otimes_k C)^{s_2}, \delta_{\widetilde{L}  } \otimes  1_C) \oplus 	 ((\widetilde{L} \otimes_k C)^r, \tilde{\sigma}_{( \widetilde{L}\otimes_k C )^r}) 
				\cong 
			\end{equation*}
			\begin{equation}\label{c8}((\widetilde{L} \otimes_k C)^{s_1}, \gamma_{\widetilde{L}} \otimes  1_C) \oplus ((\widetilde{L}\otimes_k C )^r, \tilde{\sigma}_{(\widetilde{L} \otimes_k C)^r}),
			\end{equation}
			for some $\tilde{\sigma}_{(\widetilde{L}\otimes_k C)^r}\in \Aut ((\widetilde{L} \otimes_k C)^r)$.
			Since $k$ is an algebraically closed, for any maximal ideal $\m$ of $\widetilde{L}$, we have $\frac{\widetilde{L}}{\m}=k$. 
			Then from \eqref{c8}, we have
			$$
			(Q_C , \sigma_C) \oplus(C^{s_2} ,  \delta_{k}  \otimes 1_C) \oplus (C ^r, \tilde{\sigma}_{C^r}) \cong 
			(C^{s_1} ,  \gamma_{k}  \otimes 1_C) \oplus (C ^r, \tilde{\sigma}_{C^r}),
			$$
			for some $\tilde{\sigma}_{{C}^r}  \in \Aut ({C}^r)$.
			Hence, we have 
			$$[Q_C , \sigma_C]=[C^{s_1}, \gamma_{k} \otimes 1_C]-[C^{s_2}, \delta_{k} \otimes 1_C] \text{ in }K_1(C).$$
			Thus, the map $\eta$ is surjective, and hence  $K_1(C)=k^*$.
		\end{proof}
		
		We now prove the first part of Theorem~A.
		
		\begin{thm}\thlabel{stablyp}
			Let $k$ be an algebraically closed field and $A$ be an affine domain as in \eqref{AA}. 
			Suppose that $C:=\frac{k[Z,T]}{(f)}$ is a regular domain and $A^{[l]}=k^{[m+l+2]}$ for some $l \geqslant 0$. Then $C=k^{[1]}$.
		\end{thm}
		\begin{proof}
			Note that $E:= k[x_1,\ldots,x_m] \hookrightarrow A$ is flat (\thref{flatness 2}). Consider the maps 
			\begin{tikzcd}
				k \arrow[r, hook, "\sigma"] & E \arrow[r,hook, "\gamma"] &A \arrow[r, hook, "\delta"] &A^{[l]}.
			\end{tikzcd}
			By \thref{split}, $ G_i(k) \xrightarrow{G_i(\sigma)} G_i(E)$, $ G_i(A) \xrightarrow{G_i(\delta)}~ G_i(A^{[l]})$ are isomorphisms and since $A^{[l]}=k^{[l+m+2]}$, $ G_i(k) \xrightarrow{G_i(\delta \gamma \sigma)} G_i(A^{[l]})$ are isomorphisms for every $i \geqslant 0$. Therefore, it follows that $G_i(\gamma): G_i(E) \rightarrow G_i(A)$ are isomorphisms for every $i \geqslant 0$.  
			Let $p_1,\ldots,p_n$ be distinct prime factors of $\alpha$ in $E$ and let 
			$u=\prod_{i=1}^{n} p_i$.
			Now by \thref{fcom}, we get the following commutative diagram.
			$$
			\begin{tikzcd}
				G_{j}(E) \arrow{r} \arrow[d, "\cong"] & G_{j}(E[u^{-1}])\arrow{r}\arrow[d,"\cong"] & G_{j-1}(\frac{E}{uE}) \arrow{r}\arrow[d] & G_{j-1}(E)\arrow{r}\arrow[d,"\cong"] & G_{j-1}(E[u^{-1}]) \arrow[d,"\cong"] \\
				G_{j}(A) \arrow{r} &
				G_{j}(A[u^{-1}])\arrow{r} & G_{j-1}(\frac{A}{uA}) \arrow{r} & 
				G_{j-1}(A) \arrow{r} &
				G_{j-1}(A[u^{-1}]).
			\end{tikzcd}
			$$
			By the Five lemma the canonical maps 
			\begin{equation}\label{r}
				G_i\left(\dfrac{E}{uE}\right) \rightarrow G_i\left(\frac{A}{uA}\right)
			\end{equation}
			are isomorphisms for every $i\geqslant 0$. 	Let $R=\frac{E}{uE}$. 
			Then $\dfrac{A}{uA}=\dfrac{R[Y,Z,T]}{(f)}=(R \otimes_k C)^{[1]}$. Note that $R$ is a reduced affine $k$-algebra.
			Thus from \eqref{r} and \thref{split}, it follows that the inclusion $R \hookrightarrow \frac{R[Z,T]}{(f)}$ induces canonical isomorphisms
			\begin{equation*}\label{c1}
				G_i(R) \rightarrow G_i\left( \frac{R[Z,T]}{(f)} \right)=G_i(R\otimes _k C), \text{ for every } i\geqslant 0.
			\end{equation*} 
			Therefore, by \thref{lem4}, we get that the canonical inclusion $k \hookrightarrow \dfrac{k[Z,T]}{(f)}$ induces isomorphisms 
			$\iota:K_0(k) \rightarrow ~K_0(C)$ and $\eta:K_1(k) \rightarrow K_1(C)$.
			Since $\eta$ maps ${k}^*$ onto $C^*$, we have $K_1(C)=C^*={k}^*$. 
			Now since $K_0(C) =K_0(k) =\mathbb{Z}$, we have  ${\rm Cl} (C)=0$ and hence $C$ is a PID. 
			Therefore, by \thref{alg}, we have $C=k^{[1]}$.
		\end{proof}	
		\begin{rem}
			{\rm The above result helps us to recognise the nontriviality of $A$ when $f$ is a regular domain in $\overline{k}[Z,T]$ and $\overline{k}[Z,T]/(f)\neq \overline{k}^{[1]}$.
				For instance, 
				$$A:=\dfrac{k[X_1,X_2,Y,Z,T]}{(X_1^2(1+X_2^2)^2Y -(ZT+Z+T)-X_1(1+X_2^2)h(X_1,X_2,Z,T))},\text{ for some }h\in k^{[4]},$$
				is not stably isomorphic to a polynomial ring over $k$.}
		\end{rem}
		
		\begin{cor}\thlabel{stably}
			Let $A$ and $H$ be as in \eqref{AA}.
			Suppose that $f(Z,T)= a_0(Z)+a_1(Z)T$, for some $a_0,a_1\in k^{[1]}$ and $A^{[l]}=k^{[m+l+2]}$ for some $l \geqslant 0$. Then $k[Z,T]=k[f]^{[1]}$ and $k[\X,Y,Z,T]=k[\X,H]^{[2]}$.	
		\end{cor}
		\begin{proof}
			Since $A^{[l]}=k^{[l+m+2]}$, $\overline{A}:=A\otimes_k
			\overline{k}$ is a UFD and $\overline{A}^{*} = \overline k^{*}$. Therefore by \thref{UFDline}, $f$ is irreducible in $\overline k[Z,T]$. 
			We now consider two cases:
			
			\smallskip
			
			\noindent
			{\it Case 1:}
			$a_1(Z)= 0$.\\
			Then $f(Z,T) = a_0(Z)$.
			Thus $a_0(Z)$ is irreducible in $\overline k[Z,T]$ and hence a linear polynomial in $Z$ over $k$. Therefore, $f$ is a coordinate in $k[Z,T]$. 
			
			\smallskip
			
			\noindent
			{\it Case 2:}
			$a_1(Z)\neq 0.$ \\
			We show that $a_1(Z)\in \overline{k}^{*}$, and hence $a_1(Z)\in k^{*}$. Since $f$ is irreducible in $\overline{k}[Z,T]$ we have $\gcd_{\overline{k}[Z]}(a_0(Z),a_1(Z))=1$.
			Hence $\frac{\overline{k}[Z,T]}{(f)}=\overline{k}\left[Z,\frac{1}{a_1(Z)}\right]$
			which is a regular domain.	
			Therefore, by \thref{stablyp},  $\left(\frac{\overline{k}[Z,T]}{(f)}\right)^{*}=\overline{k}^{*}$.
			Hence $a_1(Z)\in \overline{k}^{*}.$
			
			Now by \thref{RS}, we have $k[\X,Y,Z,T]=k[\X,H]^{[2]}$.
		\end{proof}
		
		We now prove the remaining part of Theorem~A. 
		
		\begin{thm}\thlabel{thmA}
			Let $k$ be a field of characteristic zero with $A$ and $H$ as in \eqref{AA}. Suppose that $A^{[l]}=k^{[l+m+2]}$ for some $l\geqslant 0$ and $k[Z,T]/(f)$ is a regular ring.
			Then $k[Z,T]=k[f]^{[1]}$ and $k[\X,Y,Z,T]=k[\X,H]^{[2]}$.
		\end{thm}
		\begin{proof}
			Let $\overline{A}=A \otimes_k \overline{k}$. Since $\overline{A}^{[l]}=\overline{k}^{[l+m+2]}$, we have $\overline{A}$ is a UFD and $\overline{A}^*=\overline{k}^*$. Therefore, by \thref{UFDline}, it follows that $f$ is irreducible in $\overline{k}[Z,T]$. 
			As $k$ is a field of characteristic zero, we have $\overline{k}[Z,T]/(f)$ is a regular domain. Thus, by \thref{stablyp}, we have $\overline{k}[Z,T]/(f)=\overline{k}^{[1]}$. Since $k$ is a field of characteristic zero, by \thref{ams} and \thref{sepco}, we have $k[Z,T]=k[f]^{[1]}$. Now, by \thref{RS}, we have $k[\X,Y,Z,T]=k[\X,H]^{[2]}$.
		\end{proof}

		Next we prove Theorem~B. 
		
		\begin{thm}\thlabel{mainch0}
			Let $k$ be a field of characteristic zero and $A$ be an affine domain as in \eqref{AA}. Let $\alpha= \prod_{i=1}^{n} p_i^{s_i}$ be the prime factorization of $\alpha$ in $k[X_1,\ldots,X_m]$. Suppose that one of the following conditions is satisfied:
			\begin{enumerate}[\rm (I)]
				\item $s_i=1$ for some $i$.
				
				\item $s_i>1$ for every $i$ and at least one of the following holds.
				\begin{itemize}
					\item [\rm (a)] $p_j^2 \mid h$ for some $j$.
					
					
					\item [\rm (b)] $(p_j, (p_j)_{X_1}, \ldots, (p_j)_{X_m}) k[X_1,\ldots,X_m]$ is a proper ideal for some $j$.
					
					\item [\rm (c)] If $n \geqslant 2$, then $(p_l,p_j)k[X_1,\ldots,X_m]$ is a proper ideal for some $l\neq j$.
				\end{itemize}
				
			\end{enumerate} 
			Then the following statements are equivalent:
			\begin{enumerate}[\rm(i)]
				
				\item  $k[\X,Y,Z,T]=k[\X,H]^{[2]}$.
				
				\item  $k[\X,Y,Z,T]=k[H]^{[m+2]}$.
				
				\item $A=k[\x]^{[2]}$.
				
				\item $A=k^{[m+2]}$.
				
				\item $k[Z,T]=k[f(Z,T)]^{[1]}$.
				
				\item  $A$ is an $\A^{2}$-fibration over $k[x_1,\ldots,x_m]$.
				
				\item  $A^{[l]}=k^{[m+l+2]}$ for some $l \geqslant 0$.
			\end{enumerate}
		\end{thm} 
		\begin{proof}
			We prove the equivalence of the statements as follows.
			$$
			
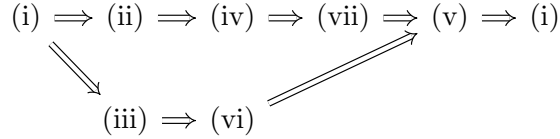
\begin{tikzcd}[column sep=small]
				{\rm(i)}\arrow[rd, Rightarrow] \arrow[r, Rightarrow] & {\rm(ii)} \arrow[r, Rightarrow]& {\rm(iv)}\arrow[r, Rightarrow] & {\rm (vii)}\arrow[r, Rightarrow]& {\rm(v)}\arrow[r, Rightarrow]& {\rm(i)} \\
				& {\rm(iii)}\arrow[r, Rightarrow] &  {\rm(vi)}\arrow[urr, Rightarrow]
			\end{tikzcd}
			$$
			Note that $\rm (i) \Rightarrow (ii) \Rightarrow (iv) \Rightarrow (vii)$ and $\rm (i) \Rightarrow (iii) \Rightarrow (vi)$ follow trivially. $\rm (v) \Rightarrow (i)$ follows from \thref{RS}. Therefore, it is enough to prove $\rm (vi) \Rightarrow (v)$ and $\rm (vii) \Rightarrow (v)$.
			
			\smallskip
			\noindent
			$\rm (vi) \Rightarrow (v):$ Let $\m$ be a maximal ideal of $E:=k[X_1,\ldots,X_m]$ containing $\alpha$. Note that $L= \frac{E}{\m}$ is a finite field extension of $k$. By \thref{Fib2}, we have 
			$\frac{L[Z,T]}{(f)}=L^{[1]}$. Since $ch.\,k=0$, by \thref{ams}, we have $L[Z,T]=L[f]^{[1]}$. Now the assertion follows by \thref{sepco}.
			
			\smallskip
			\noindent
			$\rm (vii) \Rightarrow (v):$ Since $A^{[l]}=k^{[m+l+2]}$, $A$ is a regular domain. 
			Hence by \thref{smooth2}, $\frac{k[Z,T]}{(f)}$ is a regular domain. Therefore, the result follows from \thref{thmA}.
		\end{proof}
		
		\begin{rem}
			{\rm	Let $k$ be a field of characteristic zero and $A$ be an affine domain as in \eqref{AA}, with $h=0$, i.e., $H:=\alpha(\X)Y-~f(Z,T)$.
				Then $H$ satisfies either condition (I) or condition (II)(a) of \thref{mainch0}.
				Thus we know that for the above mentioned form of $H$, the Abhyankar-Sathaye Conjecture is satisfied and moreover, $A$ is a hyperplane if and only if $f$ is a coordinate in $k[Z,T]$.}
		\end{rem}
		
		\section{On Theorems C and D}	\label{THC}	
		
		In Subsection~\ref{Exp}, we recall some basic results concerning exponential maps. In Subsection~\ref{w}, we introduce a degree function on an affine algebra defined by a linear equation (cf. \eqref{A_R}) and show its admissibility with respect to a generating set (\thref{Admissible1}). In Subsection~\ref{MTHC}, we first construct an associated graded ring (\thref{Admissibilty 2}) and use it to prove Theorem~C in two parts (Theorems~\ref{lin}~and~\ref{lin2}).
		Finally, we prove Theorem~D (\thref{main}) in Subsection~\ref{THD}.
		
		\subsection{Exponential maps and related concepts}\label{Exp}
		
		\smallskip	
		

		\medskip
		\noindent
		{\bf Definition.}
		Let $B$ be a $k$-algebra and $\phi_U: B \rightarrow B[U]$ be a $k$-algebra homomorphism. We say that $\phi = \phi_U$ is an {\it exponential map} on $B$, if $\phi$ satisfies the following two properties:
		\begin{enumerate}[\rm(i)]
			\item $\varepsilon_0\phi_U$ is identity on $B$, where $\varepsilon_0: B[U]\rightarrow B$ is the evaluation map at $U = 0$.
			
			\item $\phi_V\phi_U = \phi_{V+U}$; where $\phi_V: B\rightarrow B[V]$ is extended to a $k$-algebra homomorphism $\phi_V:B[U]\rightarrow B[U,V]$ by defining $\phi_V(U) = U$.	
		\end{enumerate}		
		
		%
		%
		%
		%
		%
		%
		%
		Given an exponential map $\phi$ on a $k$-algebra $B$, the ring of invariants of $\phi$ is a subring of $B$ given by 
		\begin{center}
			$B^{\phi} = \left\{ a\in B\,|\, \phi(a) = a\right\}$.
		\end{center}
		An exponential map $\phi$ on $B$ is said to be non-trivial if $B^{\phi} \ne B$. We denote the set of all exponential maps on $B$ by EXP$(B)$.
		The 
		$Derksen$ $invariant$ of $B$ is a subring of $B$ defined by
		\begin{center}
			$\dk(B) = k[B^{\phi}\mid \phi$ is a  non-trivial exponential map on $B]$
		\end{center}
		and the 
		{\it Makar-Limanov invariant} of $B$ is a subring of $B$ defined by
		\begin{center}
			$\ml(B) = \bigcap\limits_{\phi \in \text{EXP}(B)}B^{\phi}$.
		\end{center} 
		
		We list below some useful properties of exponential maps (cf. \cite[Chapter I]{miya}, \cite{cra} and \cite{inv}) on a $k$-domain.
		
		\begin{lem}\thlabel{lemma : properties}
			Let $\phi$ be a non-trivial exponential map on a $k$-domain $B$. Then the following holds: 	\begin{enumerate}[\rm(i)]			
				\item $B^{\phi}$ is factorially closed in $B$ i.e., for any $a,b\in B\setminus\{0\}$, if $ab\in B^{\phi}$ then $a,b\in B^{\phi}$. Hence, $B^{\phi}$ is algebraically closed in $B$.
				
				\item $\td_k(B^{\phi})$ = $\td_k(B)-1$.
				
				\item Let $S$ be a multiplicative closed subset of $B^{\phi}\setminus \{0\}$. Then $\phi$ extends to a non-trivial exponential map $S^{-1}\phi$ on $S^{-1}B$ defined by $S^{-1}\phi(a/s) = \phi(a)/s$, for all $a\in B$ and $s\in S$. Moreover, the ring of invariants of $S^{-1}\phi$ is $S^{-1}(B^{\phi})$.
				
				\item 
				$\phi$ extends to a non-trivial exponential map ${\phi}\otimes id$ on $B\otimes_k\overline{k}$ such that $(B\otimes_k \overline{k})^{{\phi}\otimes id} = B^{\phi} \otimes_k \overline{k}$.
			\end{enumerate}	
		\end{lem}
		
		
		\begin{lem}\thlabel{MLDK}
			Let $B = k^{[n]}$ then $\ml(B) = k$	and $\dk(B) = B$ for $n\geqslant 2$.
		\end{lem}
		
		
		Next we define below a {\it proper} and {\it admissible $\bZ$-filtration} on a $k$-domain $B$.
		
		\medskip
		\noindent
		{\bf Definition.} 
		A collection $\{ B_n\mid n\in \bZ \}$ of $k$-linear subspaces of $B$ is said to be a {\it proper $\bZ$-filtration} if the following properties hold:
		\begin{enumerate}[\rm(i)]
			\item $B_n\subseteq B_{n+1}$ for every $n\in \bZ$.
			
			\item $B = \cup_{n}B_n$.
			
			\item $\cap_n B_n = \{ 0 \}$.
			
			\item $(B_n \setminus B_{n-1})(B_m \setminus B_{m-1})\subseteq B_{m+n}\setminus B _{m+n -1}$ for all $m,n\in \bZ$.
		\end{enumerate}	
		
		Any proper $\Z$-filtration $\{B_n\}_{n\in \Z}$ on $B$ determines an  {\it associated $\Z$-graded integral domain}
		\begin{center}
			$\gr(B):= \bigoplus\limits_{n\in \Z} \frac{B_n}{B_{n-1}}$
		\end{center}
		and there exists a natural map $\rho: B \rightarrow \gr(B)$ defined by $\rho(a) = a + B_{n-1}$, if $a\in B_n\setminus B_{n-1}$.
		
		\begin{rem}
			\rm Note that a degree function $\deg$ on $B$ gives rise to a proper $\bZ$-filtration on $\{B_n\}_{n \in \bZ}$ on $B$ given by $B_{n}=\{b \mid \deg(b) \leqslant n\}$. 
		\end{rem}
		
		\medskip
		\noindent
		{\bf Definition.}	
		A proper $\Z$-filtration $\{B_n\}_{n\in \Z}$ on an affine $k$-domain $B$ is said to be {\it admissible} if there exists a finite generating set $\Gamma$ of $B$ such that, for any $n\in\Z$ and $a\in B_n$, $a$ can be written as a finite sum of monomials in $k[\Gamma] \cap B_n$. 
		
		\smallskip
		
		We now record a theorem on homogenization of exponential map by H. Derksen, O. Hadas and L. Makar-Limanov  \cite{dom}. The following version can be found in \cite[Theorem~2.6]{cra}.
		
		\begin{thm}\thlabel{dhm}
			Let $B$ be an affine $k$-domain with an admissible proper $\mathbb{Z}$-filtration and $gr(B)$ be the induced associated $\mathbb{Z}$-graded domain. Let $\phi$ be a non-trivial exponential map on $B$, then $\phi$
			induces a non-trivial, homogeneous exponential map $\overline{\phi}$ on $gr(B)$ such that $\rho(B^{\phi}) \subseteq gr(B)^{\overline{\phi}}$, where $\rho: B \rightarrow \gr(B)$ denotes the natural map.
		\end{thm}
		
		\subsection{On the admissibility of a $\Z$-filtration }\label{w}
		
		\smallskip
		
		Let $R$ be a UFD and 
		\begin{equation}\label{A_R}
			A_R:=\frac{R[X,Y,Z,T]}{(X^d\alpha_1(X)Y-F(X,Z,T))}\text{~, for some ~} d\geqslant 1 \text{~with~}\alpha_1(0)\neq 0,\, F(0,Z,T)\neq 0
		\end{equation}
		be a domain. Let $x,y,z,t$ denote the images of $X,Y,Z,T$ in $A_R$ respectively. 
		Then $A_R\subseteq R\left[x,\frac{1}{x},\frac{1}{\alpha_1(x)}, z,t\right]$.
		Note that every  element $r\in R\left[x,\frac{1}{x},\frac{1}{\alpha_1(x)}, z,t\right]$ is of the form 
		$$
		r  = x^{i_1}p, \text{~for some~} i_1\in \Z,\, p\in R\left[x,\frac{1}{\alpha_1(x)},z,t\right]\text{~with~} x\nmid p.
		$$
		Therefore, the function $w_R$ on $R\left[x,\frac{1}{x},\frac{1}{\alpha_1(x)}, z,t\right]$ defined by 
		$$
		w_R(r) = -i_1, \text{ for an } r \text{ as above, }
		$$ 
		is a degree function.
		Thus $w_R$ restricts to a degree function on $A_R$ and induces a proper $\Z$-filtration $\{A_R^n\}_{n\in \Z}$ on $A_R$ such that 
		$$
		A_R^n:=\{g\in A_R: w_R(g)\leqslant n\}, n\in \Z
		.$$
		Note that $y=\frac{F(x,z,t)}{x^d\alpha_1(x)}\in A_R$ and $F(0,z,t)\neq 0$, therefore, it follows that $w_R(y)=d$.
		
		In this subsection, we shall prove the following theorem.
		
		\begin{thm}\thlabel{Admissible1}
			Let $R$ be a finitely generated $k$-algebra and a UFD.
			Let $R$, $A_R$ and $w_R$ be as above with $\gcd_{R[Z,T]}(\alpha_1(0), F(0,Z,T))=1$. Let $\{c_1,\dots,c_n\}$ be a generating set of the affine $k$-algebra $R$. 
			Then $w_R$ induces an admissible $\Z$-filtration on $A_R$ with respect to the generating set $\{c_1,\dots,c_n,x,y,z,t\}$ such that
			$$gr(A_R)\cong \dfrac{R[X,Y,Z,T]}{(\alpha_1(0)X^dY-F(0,Z,T))}.
			$$ 
		\end{thm}
		\begin{proof}
			Follows from \thref{ad} and \thref{gr} proved below.
		\end{proof} 
		\begin{rem}
			{\rm Note that when $R$ is a field $k$, then the condition $\gcd_{k[Z,T]}(\alpha_1(0), F(0,Z,T))=~1$ becomes redundant as $\alpha_1(0)\in k^{*}$.}
		\end{rem}
		
		We first prove three technical lemmas (Lemmas \ref{lem1}, \ref{lem2} and \ref{lem3}).
		
		\begin{lem}\thlabel{lem1}
			Let $b \in A_R$ be an element such that $w_R(b)<0$. Suppose that	$b=\sum_{i=1}^{s}m_i$,
			where $m_i$ is a monomial in $R[x,x^dy,z,t]$ with $w_R(m_i)=0$, for all $i,1\leqslant i\leqslant s$. 
			Then $b=xb_1$, for some $b_1 \in  R[x,x^dy,z,t]$.
		\end{lem}
		\begin{proof}
			Since  $b=\sum_{i=1}^{s}m_i$, with $w_R(m_i)=0$, it follows that $m_i\in R[x^dy,z,t]$. As $w_R(b)<0$, it follows that $w_R(\alpha_1(0)^ib)<0$ for every integer $i>0$ and 
			further it follows that $\alpha_1(0)^rb\in (\alpha_1(0)x^dy - F(0,z,t))R[x^dy,z,t]$ for some $r>0$.
			Now \begin{equation}\label{F}
				\alpha_1(0)x^dy - F(0,z,t)= -(\alpha_1(x)-\alpha_1(0))x^dy + (F(x,z,t)-F(0,z,t))\in xR[x,x^dy,z,t].
			\end{equation}
			Thus, $\alpha_1(0)^rb\in  xR[x,x^dy,z,t]$.
			Note that 
			$$R[x,x^dy,z,t] \cong \dfrac{R[X,U,Z,T]}{(\alpha_1(X)U-F(X,Z,T))}.$$
			Hence 
			$R[x,x^dy,z,t]/(x) \cong R[U,Z,T]/(\alpha_1(0)U-F(0,Z,T))$. 
			Since \\$\gcd_{R[Z,T]}(\alpha_1(0),F(0,Z,T))=1$, it follows that $\{x, \alpha_1(0)\}$ is a regular sequence in $R[x,x^dy,z,t]$ and hence we get that $b \in xR[x,x^dy,z,t]$. Thus the result follows.	
		\end{proof}
		
		\begin{lem}\thlabel{lem2}
			Let $b \in R[x,x^dy,z,t] \subseteq A_R$ be an element such that $w_R(b)=-n$, for some $n\geqslant 0$. Then $b$ can be represented as a sum of monomials $m_i$ in $\{x,x^dy,z,t\}$ with $w(m_i) \leqslant -n$, for every $i$.  
		\end{lem} 
		\begin{proof}
			We show that $b=x^nb_n$, for some $b_n \in R[x,x^dy,z,t]$ with $w_R(b_n)=0$. 
			The lemma will follow as for any monomial $m=\lambda x^{\iota}(x^dy)^{\beta}z^{l}t^j$ with $\lambda\in R\setminus\{0\}$, $\iota\geqslant 0$, $\beta\geqslant0$, $l\geqslant 0$ and $j\geqslant 0$,  $w_R(m)=-\iota\leqslant 0$.
			
			We prove that $b=x^nb_n$ by induction on $n \geqslant 0$.
			Now if $n=(-w_R(b))=0$, then we are already done.
			Now consider $n>0$. We have 
			$$
			b= \sum_{i=1}^{s} m_i^{\prime} + \sum_{i=s+1}^{r} m_i^{\prime} \in R[x, x^dy,z,t],
			$$ 
			where $m_i^{\prime}$'s are monomials in $\{x, x^dy,z,t\}$ with $w_R(m_i^{\prime})<0$ for $1 \leqslant i \leqslant s$ and $w_R(m_i^{\prime})=0$ for $s+1\leqslant i
			\leqslant r$. 
			Since $w_R(b)<0$ and $w_R(\sum_{i=1}^{s} m_i^{\prime})<0$, we have 
			$w_R(\sum_{i=s+1}^{r} m_i^{\prime})<0$ and hence by \thref{lem1}, 
			$\sum_{i=s+1}^{r} m_i^{\prime}=x\widetilde{b_1}$ for some $\widetilde{b_1} \in R[x,x^dy,z,t]$.
			Also for every $i, 1 \leqslant i \leqslant s$, $x \mid m_i^{\prime}$ in $R[x,x^dy,z,t]$.
			Hence $b=x b_1$, for some $b_1 \in R[x,x^dy,z,t]$ with $w_R(b_1)=-(n-1)$. 
			Hence by induction hypothesis, we get the desired result. 
		\end{proof}

		\begin{lem}\thlabel{lem3}
			Let $b \in A_R$ be an element such that $w_R(b)=n$. If $b$ is a sum of monomials in $\{x,y,z,t\}$ each of degree $n+m$ for some $m>0$, then $x \mid b$ and $b$ has another representation as a sum of monomials in $\{x,y,z,t\}$ each of degree less than or equal to $n$. 
		\end{lem}
		\begin{proof}
			Let $b= \sum_{i=1}^{s} m_i$, where $m_i$ is a monomial in $\{x,y,z,t\}$ with $w_R(m_i)=n+m$, for some $m>0$. We first show that there exist $\iota\geqslant 0, \beta\geqslant 0$ such that $m_i=\lambda_i y^{\iota}x^{\beta} m_i^{\prime}$ for each $i$, $1\leqslant i\leqslant s$ such that $\lambda_i\in R\setminus\{0\},\,m_i^{\prime}$ is a monomial in $\{x^dy,z,t\}$, i.e.,  $w_R(m_i^{\prime})=0$ and $d\iota-\beta=n+m$.
			
			If $n+m\leqslant 0$, then we can take $\iota=0$ and $\beta=-(n+m)\geqslant 0$. 
			If $n+m>0$, then we can take $\iota$ to be the least integer greater than equal to 
			$(n+m)/d$ and $\beta=d\iota-(n+m)$. 
			Therefore, $b=y^{\iota}x^{\beta} \tilde{b}$, where $\tilde{b} \in R[x,x^dy,z,t]$ and $w_R(\tilde{b})=-m$. 
			Now the result follows by \thref{lem2}.  
		\end{proof}
		%
		We now assume that $R$ is also an affine $k$-algebra.
		
		\begin{cor}\thlabel{ad}
			Let $\{c_1,\dots,c_n\}$ be a generating set of the affine $k$-algebra $R$.
			Then the $\mathbb{Z}$-filtration $\{A_R^n\}_{n \in \mathbb{Z}}$ on $A$ is admissible with respect to the generating set $\{c_1,\dots,c_n,x,y,z,t\}$. 	
		\end{cor}
		\begin{proof}
			Let $b\in A_R$ be such that $w_R(b)=n$. Suppose 
			$$
			b=b_1+\cdots+b_r,
			$$
			such that for every $j, 1\leqslant j \leqslant r$, $b_j=\sum_{l=1}^{s_j}m_{jl}$, where $m_{jl}$ is a monomial in $\{x,y,z,t\}$ with $w_R(m_{jl})=n_j$, and $n_1<\cdots <n_r$. 
			
			If $n=n_r$, then we are done. If not, then $v_r:= w_R(b_r)<n_r$ as $n<n_r$. By \thref{lem3}, there exists a representation of $b_r$, $b_r=\sum_{l=1}^{s^{\prime}}m^{\prime}_{rl}$ with $w_R(m^{\prime}_{rl}) \leqslant v_r$. Hence
			$$
			b= \widetilde{b_1}+\cdots+\widetilde{b_{r^{\prime}}},
			$$ 
			such that for every $i, 1\leqslant i \leqslant r^{\prime}$, $\widetilde{b_i}$ is a sum of monomials of $w_R$-degree $n_i^{\prime}$, where $n_1^{\prime}< \cdots < n^{\prime}_{r^{\prime}}$ and $n^{\prime}_{r^{\prime}}<n_r$. Hence repeating the above process finitely many times we get a representation of $b$ as a sum of monomials, each of them having $w_R$-degree not more than $n$. Thus, the result follows. 
		\end{proof}
		
		The next lemma describes the structure of the associated graded ring of $A_R$ with respect to the filtration $\{A_R^n\}_{n \in \mathbb{Z}}$.
		
		\begin{lem}\thlabel{gr}
			Let $\tilde{A}:=\bigoplus_{n \in \mathbb{Z}} \frac{A_R^n}{A_R^{n-1}}$ be the associated graded ring of $A$ with respect to the filtration $\{A_R^n\}_{n \in \mathbb{Z}}$. Then 
			$$
			\tilde{A} \cong \frac{R[X,Y,Z,T]}{(\alpha_1(0)X^dY-F(0,Z,T))}.
			$$
		\end{lem}
		\begin{proof}
			Let $\{c_1,\dots,c_n\}$ be a generating set of the affine $k$-algebra $R$.
			Since $\{A_R^n\}_{n\in \bZ}$ is a proper  admissible $\mathbb{Z}$-filtration with respect to the generating set $\{c_1,\dots,c_n, x,y,z,t\}$, $\tilde{A}=R[\overline{x}, \overline{y}, \overline{z}, \overline{t}]$, where $\overline{x}, \overline{y}, \overline{z}, \overline{t}$ denote the images of $x,y,z,t$ in $\gr(A_R)$ respectively, via the natural map $\rho:A_R \rightarrow \gr(A_R)$ (cf. \cite[Remark 2.2]{inv}). 
			
			Now $\alpha_1(0),x^dy, F(0,z,t)\in A_0^R\setminus A_{-1}^R$. Since $x^d\alpha_1(x)y - F(x,z,t) =0$ in $A_R$
			by \eqref{F}, $\alpha_1(0)x^dy- F(0,z,t)\in A_{-1}^R$. Therefore, 
			$ {\alpha_1(0)}\overline x^d\overline y - F(0,\overline z,\overline t) =0$ in $\gr(A_R)$.  
			Hence we have the following onto $k$-homomorphism 
			\begin{equation*}
				\pi:\dfrac{R[X,Y,Z,T]}{(\alpha_1(0)X^dY-F(0,Z,T))}\rightarrow \gr(A_R).
			\end{equation*}
			Now both L.H.S and R.H.S are integral domains with $\td_k \gr(A_R) =\td_k R + 3< \infty$, and thus their dimensions are equal. Therefore, $\pi$ is an isomorphism.
		\end{proof}
		
		Now we note down a few useful remarks for $A_R$ with $\gcd_{R[Z,T]}(\alpha_1(0), F(0,Z,T))=1$.
		\begin{rem}\thlabel{adm}
			{\em
				\begin{enumerate}[\rm(i)]
					\item Let $m$ be a monomial of $A_R=R[x,y,z,t]$.
					If $w_R(m)<0$ then $m\in xA_R$. If $w_R(m)>0$ then $y\mid m$.
					
					\item For any $g\in A_R$, if $w_R(g)<0$ then $g=\sum_{i=1}^{n} m_i$, where $m_i$ is a monomial in $\{x,y,z,t\}$ with $w_R(m_i) \leqslant w_R(g)$. Thus $g \in xA_R$.
					
					\item For any $g\in A_R$, if $w_R(g)>0$ then 
					$\rho(y)\mid \rho(g),$ where $\rho : A_R\rightarrow \gr(A_R)$ is the natural map.
					
					\item	Let $g\in A_R$ be such that $w_R(g)=0$. Then $g \in R[x, x^dy,z,t]$. Hence it follows that 
					$$
					\alpha_1(0)^r g= \sum_{i=0}^{r} g_i(z,t) (x^dy)^i \alpha_1(0)^r +xg^{\prime},
					$$
					for some $g^{\prime} \in R[x,x^dy,z,t]$.
					Now note that 
					$$
					\alpha_1(0)x^dy=(\alpha_1(0)-\alpha_1(x))x^dy+ \alpha_1(x)x^dy=\alpha_1^{\prime}(x) x^{d+1}y+F(x,z,t),
					$$
					for some $\alpha_1^{\prime} \in R^{[1]}$. 
					Thus it follows that $\alpha_1(0)^rg= g_1^{\prime}(z,t)+ xu(x,x^dy,z,t)$, for some $r\geqslant0$, $g_1^{\prime}\in R[z,t],\, u\in R[x,x^dy,z,t]$ with $w_R(u)\leqslant 0$.
			\end{enumerate}}
		\end{rem}
		\subsection{On Theorem C}\label{MTHC}
		
		\smallskip
		
		We now introduce a definition for convenience.
		\begin{Defn}\thlabel{type A}
			{\rm 
				For ${\bf r}=(r_1,\dots,r_m)\in \Z_{>0}^m$, we say a non-zero polynomial $\alpha\in k^{[m]}$ is {\it ${\bf r}$-divisible} if there exists a system of coordinates $\{\X\}$ of $k^{[m]}$ such that 
				$$\alpha(\X) = X_1^{r_1}\alpha_1(\X),\text{ for some }\alpha_1\in k^{[m]} \text{ with }X_1\nmid \alpha_1;$$
				for any $i\in \{2,\dots,m\}$,
				$$\alpha_i(X_{i},\dots,X_m):={\alpha_{i-1}(0,X_i,\dots,X_m)}/{X_i^{r_i}}\in k[X_i,\dots,X_m]\text{ with } X_i\nmid \alpha_i$$
				and $\alpha_{m+1}:=\alpha_m(0)\in k^{*},$
			}
		\end{Defn}
		i.e., for $i\in \{1,\dots,m\}$, there exists $\beta_i\in k^{[m-i+1]}$ such that
		$$\begin{array}{lll}
			\alpha&=& X_1^{r_1}\alpha_1(\X)\\
			&=&X_1^{r_1}(X_1\beta_1(\X)+ \alpha_1(0,X_2,\dots,X_m))\\
			&=&X_1^{r_1}(X_1\beta_1(\X)+X_2^{r_2} \alpha_2(X_2,\dots,X_m))\\
			&=&X_1^{r_1}(X_1\beta_1(\X)+X_2^{r_2}(X_2\beta_2(X_2,\dots,X_m)+X_3^{r_3}\alpha_3(X_3,\dots,X_m)))\\
			&=&\dots\\
			&=& X_1^{r_1}(X_1\beta_1(\X) +\dots +X_{m-1}^{r_{m-1}}(X_{m-1}\beta_{m-1}(X_{m-1},X_m)+ X_m^{r_m}\alpha_m(X_m))\dots)\\
			&=& X_1^{r_1}(X_1\beta_1(\X) +\dots +X_{m-1}^{r_{m-1}}(X_{m-1}\beta_{m-1}(X_{m-1},X_m)+ X_m^{r_m}( X_m\beta_m(X_m)+\alpha_{m+1}))\dots).
		\end{array}$$
		\begin{rem}
			{\em Note that for ${\bf r}=(r_1,\dots,r_m)\in \Z_{>0}^m$ and a system of coordinates $\{\X\}$ in $k^{[m]}$, $\alpha$ is   ${\bf r}$-divisible if and only if $r_1$ is the highest power of $X_1$ dividing $\alpha(\X)$ and $r_j$ is the highest power of $X_j$ dividing $\alpha_{j-1}(0,X_j,\dots,X_m)$, for all $j\in\{2,\dots,m\}$.}
		\end{rem}
		
		We illustrate the above concept with two examples which will be used later.
		
		\begin{ex}\thlabel{ex1}
			{\rm $X^2(1+X)^2$ is ($2$)-divisible in $\{X\}$ in  $k^{[1]}$.}
		\end{ex}
		\begin{ex}\thlabel{ex2}
			{\rm $X_1X_2^2(X_1 + X_2^2)$ is $(1,4)$-divisible in $\{X_1,X_2\}$ in $k^{[2]}$ and is $(2,2)$-divisible in $\{X_2,X_1\}$ in $k^{[2]}$.}
		\end{ex}
		Next we recall a definition from \cite{GDA2}.
		\begin{Defn}\thlabel{GAV}
			{\rm We call an affine domain of the following form 
				\begin{equation*}
					B:=\dfrac{k [\X, Y, Z, T ]}{(X_1^{r_1}\dots X_m^{r_m}Y-F(X_1\dots,X_m,Z,T))},\, \text{ for some } r_i>1, 1\leqslant i\leqslant m,
				\end{equation*}
				with $f(Z,T):=F(0,\dots, 0,Z,T) \neq 0$ to be a {\it Generalised Asanuma domain}.}
		\end{Defn}
		Let $x_1,\dots, x_m,y,z,t$ denote the images  of $X_1,\dots,X_m,$ $Y,Z,T$ in $B$ respectively.
		Note that the coefficient of $Y$ in the defining equation of a Generalised Asanuma domain is ${\bf r}$-divisible in $\{\X\}$, where ${\bf r}=(r_1,\dots,r_m)$.
		
		We recall below a few results  from \cite{com}, \cite{adv} and \cite{adv2} on Generalised Asanuma domains.
		We first state a lemma about its Derksen invariant (\cite[Lemma 3.3]{adv}).
		\begin{lem}\thlabel{DK}
			Let $B$ be a Generalised Asanuma domain. Then $k[\x,z,t]\subseteq~\dk(B)$.
		\end{lem}
		Now we quote a technical result obtained
		from \cite{com}, \cite{adv} and \cite{adv2}.
		\begin{prop}\thlabel{com2}
			Let $B$ be a Generalised Asanuma domain. 
			Suppose $ k[\x,z,t]\subsetneqq \dk(B) $. Then the following statements hold:
			\begin{enumerate}[\rm(i)]
				\item If $m=1$ or if $k$ is an infinite field, then there exist $Z_1,T_1 \in k[Z,T]$ and $a_0,a_1\in k^{[1]}$ such that $k[Z,T]=k[Z_1,T_1]$ and $f(Z,T)=a_0(Z_1)+a_1(Z_1)T_1$.
				\item If $f$  is a line in $k[Z,T]$, i.e., $\frac{k[Z,T]}{(f)}= k^{[1]}$ then $k[Z,T]=k[f]^{[1]}$.
			\end{enumerate} 
			\begin{proof}
				(i) The case $m=1$ is done in \cite[Proposition 3.7]{com}.
				As observed in \cite[Proposition 2.10]{adv2}, the case when $k$ is an infinite field follows from \cite[Proposition 3.4(i)]{adv}.
				
				\smallskip
				\noindent
				(ii) The result follows from \cite[Proposition 3.4]{adv}, when $k$ is infinite and \cite[Proposition 2.12]{adv2}, when $k$ is a finite field.
			\end{proof}
		\end{prop}

		Next we record another technical lemma from \cite[Lemma 3.13]{adv2}.
		
		\begin{lem}\thlabel{GA2}
			Let $B$ be a Generalised Asanuma domain. 
			Suppose $m > 1$ and $k[\x, z, t]\subsetneqq \dk(B)$.
			Then there exist an
			integer $l \in \{1,\dots, m\}$ and an integral domain $B_l$ of the form
			$$B_l= \dfrac{k(X_l)[X_1 ,\dots, X_{l-1} , X_{l+1} ,\dots, X_m , Y, Z, T]}{
				(X_1^{r_1}\dots X_m^{r_m}Y - f (Z, T))}$$
			such that $\dk(B_l ) = B_l$.
		\end{lem}
		
		The next proposition illustrates that using suitable degree functions, \thref{Admissible1} can be applied successively  
		to a ring as in \eqref{AA}, with $\alpha$ being a ${\bf r}$-divisible polynomial, to obtain  a Generalised Asanuma domain as an associated graded domain.
		
		\begin{prop}\thlabel{Admissibilty 2}
			Let
			$$A= \dfrac{k[X_1,\dots,X_m,Y,Z,T]}{(\alpha(X_1,\dots,X_m)Y -f(Z,T)- h(X_1,\dots, X_m,Z,T))},$$
			be a domain such that
			\begin{enumerate}[\rm(a)]
				\item For ${\bf r}=(r_1,\dots,r_m)\in \Z_{>0}^m$, $\alpha(\X) $ is ${\bf r}$-divisible in the system of coordinates $\{\X\}$ in $k^{[m]}$.
				\item $X_1\mid h(\X,Z,T)$ in $k[\X,Z,T]$. 
			\end{enumerate}
			Then we can define a degree function $w_j$ on $A_{j-1}$ such that $w_j$ induces an admissible $\Z$-filtration on $A_{j-1}$, where $A_0:= A$ and $A_j:= \gr(A_{j-1})$ for each $j,\, 1\leqslant j\leqslant m$, such that 
			$$A_m\cong \dfrac{k[X_1,\dots,X_m,Y,Z,T]}{(X_1^{r_1}X_2^{r_2}\dots X_m^{r_m}Y - f(Z,T))}.$$
		\end{prop}
		\begin{proof}
			Since $A_0=A$ is an affine domain, we have $f(Z,T)\neq 0$.
			Now $\alpha_0:=\alpha$ is ${\bf r}$-divisible in the system of coordinates $\{\X\}$ in $k^{[m]}$. Hence we have
			$$\alpha_0(\X) = X_1^{r_1}\alpha_1(\X),\,  X_1\nmid \alpha_1(X_1,\dots,X_m),$$ 
			$$
			\alpha_{i-1}(0,X_i,\dots,X_m)= X_i^{r_i}\alpha_i(X_i,\dots,X_m),\, X_i\nmid \alpha_i(X_i,\dots,X_m) \text{ for }2\leqslant i\leqslant m
			$$
			and
			$$\alpha_{m+1}:=\alpha_m(0)\in k^{*}.$$
			
			\noindent
			Therefore, $\alpha_1(0,X_2,\dots,X_m)= X_2^{r_2}\alpha_2(X_2,\dots,X_m)$ is  ${\bf \tilde{r}}$-divisible in the system of coordinates $\{X_2,\dots,X_m\}$, where ${\bf \tilde{r}}=(r_2,\dots,r_m)$.
			Let $x_{10},\dots,x_{m0},y_0,z_0,t_0$ denote the images of $\X,Y,Z,T$ in $A_0$ respectively. 
			We first consider the degree function $w_1$ on $A_0$ defined by 
			$$w_1(x_{10})=-1,\, w_1(y_0)= r_1,\, w_1(z_0)= w_1(t_0)= w_1(x_{i0})=0 \text{ for } 2\leqslant i\leqslant m.$$
			Since $X_1\mid h(\X,Z,T)$ in $k[\X,Z,T]$ and $\gcd(\alpha_1(0,X_2,\dots,X_m),f(Z,T))=\gcd(X_2^{r_2}\alpha_2(X_2,\dots,X_m), f(Z,T))=1$, by \thref{Admissible1}, $w_1$ induces an admissible $\Z$-filtration on $A_0$ with respect to the generating set $\{x_{10},\dots,x_{m0},y_0,z_0,t_0\}$ such that the associated graded domain is
			$$
			A_1=\gr(A_0) =  \dfrac{k[\X,Y,Z,T]}{(X_1^{r_1}X_2^{r_2}\alpha_2(X_2,\dots,X_m)Y - f(Z,T))}.
			$$
			Let $x_{11},\dots,x_{m1},y_1,z_1,t_1$ denote the images of $\X,Y,Z,T$ in $A_1$ respectively.
			We consider the function $w_2$ on $A_1$ defined as follows:
			$$w_2(x_{21})=-1,\, w_2(y_1)= r_2,\,  w_2(z_1) = w_2(t_1)=w_2(x_{i1}) =0,\,\text{ for } i\in\{1,3,\dots,m\}.$$ 
			Since $\gcd(X_1^{r_1}\alpha_2(0,X_3,\dots,X_m), f(Z,T))= \gcd(X_1^{r_1}X_3^{r_3}\alpha_3(X_3,\dots,X_m),f(Z,T))=1$, by \thref{Admissible1}, $w_2$ 
			induces an admissible $\Z$-filtration on $A_1$
			such that the associated graded domain is
			$$
			A_2=\gr(A_1)=
			\dfrac{k[X_1,\dots,X_m,Y,Z,T]}{(X_1^{r_1}X_2^{r_2}X_3^{r_3}\alpha_3(X_3,\dots,X_m)Y -f(Z,T))}.
			$$
			Thus proceeding successively for any $j\in \{2,\dots,m\}$, we get domains
			$$
			\begin{array}{lll} A_{j-1}
				&=& \dfrac{k[X_1,\dots,X_m,Y,Z,T]}{(X_1^{r_1}\dots X_{j}^{r_{j}}\alpha_{j}(X_j,\dots,X_m)Y -f(Z,T))}.
			\end{array}
			$$
			Let $l=j-1$ and $x_{1l},\dots,x_{ml},y_l,z_l,t_l$ denote the images of $\X,Y,Z,T$ in $A_{j-1}$ respectively.
			Now $w_j$ is a degree function on $A_{j-1}$ defined as follows:
			$$w_j(x_{jl})=-1,\, w_j(y_l)= r_j,\,  w_j(z_l) = w_j(t_l)=w_j(x_{il}) =0,\, i\in\{1,\dots,m\}\setminus\{j\}.$$ 
			Hence, by \thref{Admissible1}, $w_j$ induces an admissible $\bZ$-filtration on $A_{j-1}$ for which the associated graded domain is of the form
			$$
			A_j=\gr(A_{j-1})
			= \dfrac{k[X_1,\dots,X_m,Y,Z,T]}{(X_1^{r_1}\dots X_{j+1}^{r_{j+1}}\alpha_{j+1}(X_{j+1},\dots,X_m)Y -f(Z,T))}.
			$$
			Since $\alpha_{m+1} \in k^*$, at the $m$th step we get the domain 
			$$A_m\cong \dfrac{k[X_1,\dots,X_m,Y,Z,T]}{(X_1^{r_1}X_2^{r_2}\dots X_m^{r_m}Y - f(Z,T))}$$
			as an associated graded domain of $A_{m-1}$ with respect to the degree function $w_m$. 
		\end{proof}
		We now establish a more general version of Theorem~C.
		For convenience we split the result in two parts.
		We now prove Theorem~C for the case $\ml(A)=A$.
		
		\begin{thm}\thlabel{lin}
			Let $A$ be as in \thref{Admissibilty 2} with ${\bf r}\in \mathbb{Z}_{>1}^{m}$.
			Suppose $\ml(A)= k$. Then the following statements hold:
			\begin{enumerate}[\rm(i)]
				\item There exists an integral domain $\tilde{A}$ of the form
				$$
				\frac{k[\X,Y,Z,T]}{(X_1^{r_1}X_2^{r_2}\dots X_m^{r_m}Y-f(Z,T))}
				$$
				such that $\dk(\tilde{A})=\tilde{A}$.
				\item When $m=1$ or $k$ is an infinite field then there exist a system of coordinates $\{Z_1,T_1\}$ of $k[Z,T]$ and $a_0,a_1 \in k^{[1]}$, such that $f(Z,T)=a_0(Z_1)+a_1(Z_1)T_1$. 
				
				\item When $f$ is a line in $k[Z,T]$, then $k[Z,T]=k[f]^{[1]}$.
			\end{enumerate}
		\end{thm}
		\begin{proof}
			Note that since $A$ is a domain, $f(Z,T)\neq 0$.	
			Let $x_1,\ldots,x_m,y,z,t$ denote the images of $X_1,\ldots,X_m,Y,Z,T$ in $A$ respectively.
			Since $\ml(A)= k$, there exists a non-trivial exponential map $\phi$ on $A$ such that $x_1 \notin A^{\phi}$.
			
			Let $w_1$ be the degree function on $A$ defined by
			$$
			w_1(x_1) = -1, w_1(y)=r_1, w_1(z)= w_1(t)=w_1(x_i)=0,\, i\in \{2,\dots,m\}.
			$$
			Since $\alpha$ is ${\bf r}$-divisible in the system of coordinates $\{\X\}$, we have  $\alpha=X_1^{r_1}\alpha_1$, for some $\alpha_1\in k^{[m]}$. 
			Since $X_1\mid h$ and $\gcd(\alpha_1(0,X_2,\dots,X_m), f(Z,T))=1$, by \thref{Admissible1}, the degree function $w_1$ on $A$ induces an admissible $\Z$-filtration on ${A}$ with respect to $\{\x,y,z,t\}$ such that  
			\begin{equation*}
				A_1 :=\gr(A)\cong \dfrac{{k}[\X,Y,Z,T]}{(X_1^{r_1}\alpha_1(0,X_2,\dots,X_m)Y -f(Z,T))}.
			\end{equation*}
			For $a \in A$, let $\overline{a}$ denotes its image in $A_1$. 
			By \thref{dhm}, $\phi$ induces a non-trivial exponential map $\phi_1$ on  ${A}_1$.  We shall show that $\overline{y}\in A_1^{\phi_{1}}$. 
			
			We first note that $w_1(a) \geqslant 0$ for all $a \in A^{\phi}$. 
			Suppose not. That means there exists $a \in A^{\phi}$ such that $w_1(a)<0$. 
			Then $x_1\mid a$ in $A$ (cf. \thref{adm}(ii)) and hence $x_1 \in A^{\phi}$ (cf. \thref{lemma : properties}(i)) which contradicts the choice of $\phi$. 
			
			
			Suppose, if possible, $w_1(a)=0$ for all $a\in A^{\phi}$.
			Since $\td_{k} A^{\phi} =m+1$ (cf. \thref{lemma : properties}(ii)), there exist $m+1$ algebraically independent elements $f_1,\dots,f_{m+1} \in A^{\phi}$ with $w_1(f_i)=0,$ for $ 1\leqslant i\leqslant m+1$.
			Set $\beta:=\alpha_1(0,x_2,\ldots,x_m)$.
			By \thref{adm}(iv), for each $i$, $1 \leqslant i \leqslant m+1$, there exists an integer $s_i \geqslant 0$ such that
			$$
			\beta^{s_i} f_i= p_{i0}(x_2,\ldots,x_m,z,t)+x_1 q_i(x_1,\ldots,x_m,y,z,t), \text{ where }w_1(q_i) \leqslant 0.
			$$
			Then for any $i\in \{1,\dots,m+1\}$, 
			\begin{equation}\label{b}	
				\overline{\beta}^{s_i}\overline{f_i}=  p_{i0}(\overline{x_2},\dots ,\overline{x_{m}},\overline{z},\overline{t}),\text{ i.e., } \overline{f_i}=\dfrac{{p_{i0}}(\overline{x_2},\ldots, \overline{x_m},\overline{z},\overline{t})}{\overline{\beta}^{s_i}}		
			\end{equation}
			and $\overline{f_i}\in A_1^{\phi_1}$ for $ 1\leqslant i \leqslant m+1$ (cf. \thref{dhm}).  
			Thus
			$$k[\overline{f_1},\dots,\overline{f_{m+1}}]\subseteq D:=k\left[\overline{x_2},\dots ,\overline{x_{m}},\overline{z},\overline{t},\frac{1}{\overline{\beta}} \right].$$
			
			\noindent
			Suppose, if possible, that $\{\overline{f_1},\dots,\overline{f_{m+1}}\}$ is an algebraically independent set. 
			Since 
			$$
			k^{[m+1]}=k[\overline{x_2},\dots,\overline{x_m},\overline{z},\overline{t}]\hookrightarrow A_1\hookrightarrow k\left[\overline{x_1},\overline{x_2}\dots,\overline{x_m},\overline{z},\overline{t},\frac{1}{\overline{x_1}^{r_1}\alpha_1(0,\overline{x_2},\dots,\overline{x_m})}\right]
			$$
			and
			$$\overline{x_1}^{r_1}\alpha_1(0,\overline{x_2},\dots,\overline{x_m})\overline{y} = {f(\overline{z},\overline{t})},$$
			we have,  $\td_k D = m+1.$ 
			Therefore, $D$ is algebraic over ${k}[\overline{f_1},\dots, \overline{f_{m+1}}]$ and hence ${k}[\overline{x_2},\dots ,\overline{x_{m}},\overline{z},\overline{t}] \subseteq A_1^{\phi_{1}}$ 
			(cf. \thref{lemma : properties}(i)).
			Thus $\overline{x_1}, \overline{y} \in A_1^{\phi_1}$, which contradicts the fact that $\phi_1$ is non-trivial. 
			Therefore, there exists a non-zero $Q \in k^{[m+1]}$ such that $Q(\overline{f_1},\dots, \overline{f_{m+1}})=0$.
			Hence it follows that 
			\begin{equation}\label{q}
				Q\left(\frac{{p_{10}}(\overline{x_2},\ldots, \overline{x_m},\overline{z},\overline{t})}{\overline{\beta}^{s_1}},\ldots, \dfrac{{p_{(m+1)0}}(\overline{x_2},\ldots, \overline{x_m},\overline{z},\overline{t})}{\overline{\beta}^{s_{m+1}}}\right)=0.
			\end{equation}
			Now since $f_i=\frac{{p_{i0}}(x_2,\ldots,x_m,z,t)}{\beta^{s_i}}+ \frac{x_1q_i}{\beta^{s_i}}$, using \eqref{q}, we have $Q(f_1,\ldots,f_{m+1})=\frac{x_1 q}{\beta^s}$ for some $q \in A$ and $s \geqslant 0$, i.e, $x_1q \in \beta^s A$. Now since
			$$
			\dfrac{A}{x_1A} =\dfrac{k[X_2,\ldots,X_m,Y,Z,T]}{(f(Z,T))},
			$$
			$\beta$ is a non-zero divisor in $\frac{A}{x_1A}$. Thus it follows that $Q(f_1,\ldots,f_{m+1})=x_1\widetilde{q}$ for some $\widetilde{q} \in A$. Since $Q(f_1,\ldots,f_{m+1})\in A^{\phi}$, it implies that $x_1 \in A^{\phi}$ (cf. \thref{lemma : properties}(i)),  a contradiction to the choice of $\phi$. 
			
			Therefore, there exists an $a \in A^{\phi}$ with $w_1(a)>0$. Hence $\overline{y}\mid \overline{a}$ (cf. \thref{adm}(iii)) and thus $\overline{y}\in A_1^{\phi_1}$.

			By \thref{Admissibilty 2}, there exists an admissible $\Z$-filtration induced by degree function $w_j$ on $A_{j-1}, 1\leqslant j\leqslant m$, where $A_0 := {A},\, A_i := \gr(A_{i-1}),\, 1 \leqslant i\leqslant m$, such that  
			\begin{equation*}
				\tilde{A} := A_m\cong \frac{k[\X,Y,Z,T]}{(X_1^{r_1}X_2^{r_2}\dots X_m^{r_m}Y-f(Z,T))}.
			\end{equation*} 
			For each $a \in A$, let $a^{\prime}$ denotes its image in $\tilde{A}$. 
			Since $\overline{y}\in A_1^{\phi_1}$, by repeated application of \thref{dhm}, we see that $\phi_1$ induces a non-trivial exponential map $\tilde{\phi}$ on  $\tilde{A}$ such that $y^{\prime}\in \tilde{A}^{\tilde{\phi}}$. 
			Hence by \thref{DK}, 
			$\dk(\tilde{A})=\tilde{A}$.
			Thus the first statement is proved.
			Hence by \thref{com2}, the last two  statements follow.
	\end{proof}
	We now prove Theorem C for the case $\dk(A)=A$.
	In fact we prove a more general statement where the hypothesis $\dk(A)=A$ is replaced by a weaker condition that $\dk(A)$ contains an element of positive degree with respect to a certain degree function.
	\begin{thm}\thlabel{lin2}
		Let $A$ be as in \thref{Admissibilty 2} with ${\bf r}\in \mathbb{Z}_{>1}^{m}$.
		Let $w_1$ be the degree function on $A$ defined by $w_1(x_1)=-1, w_1(y)= r_1,w_1(x_i)= w_1(z) = w_1(t) =0,\, 2\leqslant i\leqslant m$.
		Suppose $\dk(A)$ contains an element with positive $w_1$-degree (in particular, $\dk(A)=A$). Then the following statements hold:
		\begin{enumerate}[\rm(i)]
			\item There exists an integral domain $\tilde{A}$ of the form
			$$
			\frac{k[\X,Y,Z,T]}{(X_1^{r_1}X_2^{r_2}\dots X_m^{r_m}Y-f(Z,T))}
			$$
			such that $\dk(\tilde{A})=\tilde{A}$.
			\item When $m=1$ or $k$ is an infinite field then there exist a system of coordinates $\{Z_1,T_1\}$ of $k[Z,T]$ and $a_0,a_1 \in k^{[1]}$, such that $f(Z,T)=a_0(Z_1)+a_1(Z_1)T_1$. 
			\item When $f$ is a line in $k[Z,T]$, then $k[Z,T]=k[f]^{[1]}$.
		\end{enumerate}
	\end{thm}
	\begin{proof} 
		Let $x_1,\ldots,x_m,y,z,t$ denote the images of $X_1,\ldots,X_m,Y,Z,T$ in $A$ respectively.
		As $\alpha$ is ${\bf r}$-divisible with respect to $\{\X\}$, we have $\alpha=X_1^{r_1}\alpha_1$ for some $\alpha_1\in k^{[m]}$. 
		Since $X_1\mid h$ and $\gcd(\alpha_1(0,X_2,\dots,X_m), f(Z,T))=1$,
		by \thref{Admissible1}, the degree function $w_1$ on $A$, defined by $$w_1(x_1) = -1, w_1(y)=r_1, w_1(z)= w_1(t)=w_1(x_i)=0,\, i\in \{2,\dots,m\},$$
		induces an admissible $\Z$-filtration on ${A}$ with respect to $\{\x,y,z,t\}$ such that  
		\begin{equation*}
			A_1 :=\gr(A)\cong \dfrac{{k}[\X,Y,Z,T]}{(X_1^{r_1}\alpha_1(0,X_2,\dots,X_m)Y -f(Z,T))}.
		\end{equation*}
		For $a \in A$, let $\overline{a}$ denote its image in $A_1$.
		Since $\dk(A)$ contains an element of positive $w_1$-degree, there exist a non-trivial exponential map $\phi$ on $A$ and $a\in A$ with $w_1(a)>0$ such that $a\in A^{\phi}$. 
		Hence $\overline{y}\mid \overline{a}$ (cf. \thref{adm}(iii)).
		By \thref{dhm}, $\phi$ induces a non-trivial exponential map $\phi_1$ on $A_1$ such that $\overline a \in A_1^{\phi_1}$. Therefore $\overline{y}\in A_1^{\phi_1}$ (cf. \thref{lemma : properties}(i)). 
		By \thref{Admissibilty 2}, there exists an admissible $\Z$-filtration induced by the degree function $w_j$ on $A_{j-1}, 1\leqslant j\leqslant m$, where $A_0 := {A},\, A_i := \gr(A_{i-1}),\, 1\leqslant i\leqslant m$ such that 
		\begin{equation*}
			\tilde{A} := A_m\cong \frac{k[\X,Y,Z,T]}{(X_1^{r_1}X_2^{r_2}\dots X_m^{r_m}Y-f(Z,T))} = k[\tilde{x}_1,\dots,\tilde{x}_m,\tilde{y},\tilde{z},\tilde{t}], r_i\geqslant 2,\, 1\leqslant i \leqslant m,
		\end{equation*} 
		where for each $a \in A$, $\tilde{a}$ denotes its image in $\tilde{A}$.
		By repeated application of \thref{dhm}, $\phi_1$ induces a non-trivial exponential map $\tilde\phi$ on $\tilde{A}$ with $\tilde{y}\in \tilde{A}^{\tilde{\phi}}$. 
		Hence by \thref{DK},
		$\dk(\tilde{A})=\tilde{A}$. 
		Thus the first statement is proved.
		Now by \thref{com2}, the last two  statements follow.
	\end{proof}

	As a consequence we see that if $f$ cannot be made linear by change of coordinates, then the ring $A$ is non-trivial.
	More precisely:
	
	\begin{cor}\thlabel{notpoly}
		Let $A$ be as in \thref{Admissibilty 2} with ${\bf r}\in \mathbb{Z}_{>1}^{m}$.	
		Suppose that $f(Z,T)$ has the property that for any $Z_1,T_1\in \overline{k}[Z,T]$ with $\overline{k}[Z,T]=\overline{k}[Z_1,T_1]$ there does not exist $a_0,a_1\in \overline{k}^{[1]}$ for which $f(Z,T)= a_0(Z_1)+a_1(Z_1)T_1$. Then $\ml(A)\neq k$ and $\dk(A)\neq A$. In particular, $A\not\cong k^{[m+2]}$.
	\end{cor}
	\begin{proof}
		Suppose, if possible, that $\ml(A)=k$ or $\dk(A)=A$. 
		Let $\overline{A}=A\otimes_k\overline{k}$.
		Then $\ml(\overline{A})=\overline{k}$ or $\dk(\overline{A})=\overline{A}$.
		Therefore, by \thref{lin}(ii) or \thref{lin2}(ii), we have $f=a_0(Z_1)+a_1(Z_1)T_1$, for some $a_0,a_1\in \overline{k}^{[1]}$ and $\overline{k}[Z,T]=\overline{k}[Z_1,T_1]$, a contradiction to the given hypothesis. 
		Thus $\ml(A)\neq k$ and $\dk(A)\neq A$. Hence $A\not\cong k^{[m+2]}$ (cf. \thref{MLDK}).
	\end{proof}
	\begin{rem}\thlabel{EXa}
		{\rm 
			\thref{notpoly} enables one to construct (or recognise) a large class of affine domains of Russell-Koras type: rings which are regular, factorial, contractible and satisfy many  properties of polynomial rings, without being a polynomial ring itself.
			For instance, consider 
			$$A_1:=\dfrac{k[X,Y,Z,T]}{(X^2(X+1)^2Y-(Z^2+T^3)-Xh_1(X,Z,T))}, \text{ for some } h_1\in k^{[3]}
			$$
			and 
			$$A_2:=\dfrac{k[X_1,X_2,Y,Z,T]}{(X_1X_2^2(X_1+X_2^2)Y -(Z^2+T^3)-X_2h_2(X_1,X_2,Z,T))},\text{ for some }h_2\in k^{[4]}.
			$$
			By a suitable choice of $h_1$ and $h_2$, the rings $A_1$ and $A_2$ can be made to satisfy several properties of polynomial rings. However, as $Z^2 + T^3$ cannot be linear in any system of coordinates,  Examples~\ref{ex1}~and~\ref{ex2} and \thref{notpoly} show that both the Makar-Limanov invariants and the Derksen invariants of $A_1$ and $A_2$ are non-trivial; in particular, they are not polynomial rings.
			
	}\end{rem}

	Using the arguments in \thref{Admissibilty 2} and Theorems~\ref{lin}~and~\ref{lin2}, a more general but technical version of Theorem~C has been established by the third author in her thesis (\cite{APalT}).
	For the interested reader, we give the precise statement below.
	\begin{thm}\thlabel{THGC}
		Let $k$ be an infinite field.
		Let 
		$$
		A=\dfrac{k[\X,Y,Z,T]}{(\alpha(\X)Y-F(\X,Z,T))}
		$$
		be a domain such that 
		\begin{enumerate}[\rm(a)]
			\item $f(Z,T):=F(0,\dots,0,Z,T)\neq 0$.
			\item For ${\bf r}=(r_1,\dots,r_m)$, $\alpha$ is ${\bf r}$-divisible with respect to $\{\X\}$ in $k^{[m]}$ and 	$\alpha_1 := \alpha(\X)/X_1^{r_1}$, $
			\alpha_i := \alpha_{i-1}(0,X_i,\dots,X_m)/X_i^{r_i},\, 2\leqslant i\leqslant m$. 
			\item $r_i>1$, for each $i\in \{1,\dots,m\}$.
			\item $\gcd(\alpha_i(0,X_{i+1},\dots,X_m), F(0,\dots,0,X_{i+1},\dots,X_m,Z,T))=1$, for all $i\in \{1,\dots,m\}$.
			\item Either $\ml(A)= k$ or $\dk(A)=A$. 
		\end{enumerate} 
		Then there exist a system of coordinates $\{Z_1,T_1\}$ of $k[Z,T]$ and $a_0,a_1 \in k^{[1]}$, such that $f(Z,T)=a_0(Z_1)+a_1(Z_1)T_1$. 
		Furthermore, if $f$ is a line in $k[Z,T]$ i.e., $k[Z,T]/(f)~=~k^{[1]}$, then $k[Z,T]=k[f]^{[1]}$.
	\end{thm}
	
	It follows from the above theorem that the variety defined by 
	$$
	H=X_1X_2^2(X_1+X_2^2)Y -(Z^2+T^3)-h_3(X_1,X_2,Z,T),\, h_3\in k^{[4]},\, h_3(0,0,Z,T)=0.
	$$
	is not an affine space.

	\subsection{On Theorem D}\label{THD}
	
	\smallskip
	
	In this subsection we prove an extended version of Theorem~D for an affine domain $A$ defined as follows
	\begin{equation}\label{AAA}
		A = \dfrac{k[\X,Y,Z,T]}{(\alpha(X_1,\dots,X_m)Y -f(Z,T)- h(X_1,\dots, X_m,Z,T))},
	\end{equation}
	where $H:=\alpha(X_1,\dots,X_m)Y -f(Z,T)- h(X_1,\dots, X_m,Z,T)$ such that
	\begin{enumerate}[\rm(a)]
		\item
		For ${\bf r}=(r_1,\dots,r_m)\in \Z_{>1}^m$,
		$\alpha$ is ${\bf r}$-divisible in the system of coordinates $\{X_1-\lambda_1,\dots,X_m-\lambda_m\}$, for some $\lambda_i\in \overline{k},\mi$.
		\item Every prime divisor of $\alpha(\X)$ divides $h$ in $k[\X,Z,T]$.
	\end{enumerate}
	Let $x_1,\ldots,x_m,y,z,t$ denote the images of $X_1,\ldots,X_m,Y,Z,T$ in $A$ respectively.
	Since $A$ is a domain $f(Z,T)\neq 0$.	
	We now prove 
	a crucial step towards proving Theorem~D.	
	
	\begin{prop}\thlabel{line}
		Let $A$ be an affine domain as in \eqref{AAA} and 
		$k_1=k(\lambda_1,\dots,\lambda_m)$.
		Suppose that $A^{[l]}=k^{[l+m+2]}$ for some $l\geqslant 0$ and either $\ml(A)= k$ or $\dk(A)=A$.
		Then $k_1[Z,T]=k_1[f]^{[1]}$. 
	\end{prop}
	
	\begin{proof}
		Since $\alpha$ is ${\bf r}$-divisible with respect to $\{X_1-\lambda_1,\dots,X_m-\lambda_m\}$ of $k_1^{[m]}$, for each $i\in \{1,\dots,m\}$, there exists $\alpha_i\in k_1^{[m-i+1]}$ such that
		$$
		\alpha(\X) = (X_1-\lambda_1)^{r_1}\alpha_1(X_1-\lambda_1,\dots, X_m-\lambda_m),\, (X_1-\lambda_1)\nmid \alpha_1,
		$$ 
		and, for $2\leqslant i\leqslant m$,
		$$
		{\alpha_{i-1}(0,X_i-\lambda_i,\dots,X_m-\lambda_m)}={(X_i-\lambda_i)^{r_i}}\alpha_i(X_{i}-\lambda_i,\dots,X_m-\lambda_m),\, (X_i-\lambda_i)\nmid \alpha_i.
		$$ 
		Therefore, without loss of generality, we can assume that
		$$A^{\prime}:=
		A\otimes_k k_1 \cong\dfrac{{k_1}[\X,Y,Z,T]}{(X_1^{r_1}{{\alpha_1}}(\X)Y- f(Z,T)-h_1(\X,Z,T))},
		$$ 
		where $h_1(\X,Z,T) := h(X_1+\lambda_1,\dots, X_m+\lambda_m, Z,T)$. 
		Note that now $X_1^{r_1}{\alpha_1}$ is ${\bf r}$-divisible in the system of coordinates $\{\X\}$ in $k_1^{[m]}$.
		Now $\ml(A^{\prime})= k_1$ or $\dk(A^{\prime})=A^{\prime}$ according as $\ml(A)= k$ or $\dk(A)=A$.
		
		Suppose $k$ is an infinite field. 
		Then by \thref{lin}(ii) or \thref{lin2}(ii) according as $\ml(A^{\prime})= k_1$ or $\dk(A^{\prime})=A^{\prime}$, there exist a system of coordinates $\{Z_1,T_1\}$ of $k_1[Z,T]$ and $a_0,a_1\in k_1^{[1]}$ such that $f=a_0(Z_1)+a_1(Z_1)T_1$.
		Also note that $(A^{\prime})^{[l]}=k_1^{[l+m+2]}$.
		Therefore, by \thref{stably}, we have
		$k_1[Z,T]=k_1[f]^{[1]}$.
		
		When $k$ is a finite field, then from the above paragraph, we have $\overline{k}[Z,T] = \overline{k}[f]^{[1]}$.  Therefore, by \thref{sepco}, ${k}[Z,T] = {k}[f]^{[1]}$ and hence $k_1[Z,T]=k_1[f]^{[1]}$.
	\end{proof}
	
	Next we prove an easy lemma. 
	
	\begin{lem}\thlabel{linear}
		Let $f= a_0(Z) + a_1(Z)T$, for some $a_0,a_1\in k^{[1]}$, be an irreducible polynomial of $k[Z,T]$ with $\left( \frac{k[Z,T]}{(f)} \right)^{*}=k^{*}$. Then $k[Z,T] = k[f]^{[1]}$. In particular, if $\frac{k[Z,T]}{(f)} = k^{[1]}$, then $k[Z,T] = k[f]^{[1]}$.	
	\end{lem}
	\begin{proof}
		Consider the following two cases.
		
		\medskip
		\noindent
		{\it Case} 1: If $a_1(Z)=0$, then $f=a_0(Z)$ is an irreducible polynomial. As $\left( \frac{k[Z,T]}{(f)} \right)^{*}=k^{*}$, it follows that $f$ is linear in $Z$, and hence $k[Z,T]=k[f]^{[1]}$.
		
		\medskip
		\noindent
		{\it Case} 2: If $a_1(Z) \neq 0$, then $\gcd(a_0,a_1)=1$  in $k[Z]$ as $f$ is irreducible in $k[Z,T]$.
		Hence $\frac{k[Z,T]}{(f)}=k\left[Z, \frac{1}{a_1(Z)}\right]$.
		Now since $\left(\frac{k[Z,T]}{(f)}\right)^{*}=k^{*}$, it follows that $a_1(Z) \in k^{*}$ and hence $f$ is linear in $T$. Therefore $k[Z,T]=k[f]^{[1]}$.	
	\end{proof}
	We now prove an extended version of Theorem~D.
	
	\begin{thm}\thlabel{main}
		Let $A$ and $H$ be as in \eqref{AAA}
		and for each  $i,\, \mi$,  $\lambda_i$ is separable over $k$, i.e., $k_1:=k(\lambda_1,\dots,\lambda_m)$ is separable over $k$. 
		Let $E:=k[\x]$ be a subring of $A$.		
		Then the following statements are equivalent.
		\begin{enumerate}[\rm(i)]
			
			\item  $k[\X,Y,Z,T]=k[\X,H]^{[2]}$.
			
			\item  $k[\X,Y,Z,T]=k[H]^{[m+2]}$.
			
			\item $A=k[\x]^{[2]}=E^{[2]}$.
			
			\item $A=k^{[m+2]}$.
			
			\item $k[Z,T]=k[f(Z,T)]^{[1]}$.
			
			\item	$A^{[l]}=k^{[l+m+2]}$ for some $l \geqslant 0$ and $\ml(A)=k$.
			
			\item $f(Z,T)$ is a line in $k[Z,T]$ and $\ml(A)=k$.
			
			\item $A$ is an $\mathbb{A}^{2}$-fibration over $E$ and $\ml(A)=k$.	
			
			\item
			\begin{enumerate}[\rm(a)]
				\item When $m=1$, $A$ is a UFD, $\ml(A)=k$ and $\left(\frac{ k_1[Z,T]}{(f(Z,T))}\right)^{*} = {k_1}^{*}.$
				
				\item When $m>1$, $A\otimes_k k_1$ is a UFD, $\ml(A)=k$ and $\left(\frac{ k_1[Z,T]}{(f(Z,T))}\right)^{*} = {k_1}^{*}.$
			\end{enumerate}
			
			\item $A^{[l]}=k^{[l+m+2]}$ for some $l \geqslant 0$ and $\dk(A)=A$.
			
			\item $f(Z,T)$ is a line in $k[Z,T]$ and $\dk(A) = A$.	
			
			\item $A$ is an $\mathbb{A}^{2}$-fibration over $E$ and $\dk(A)=A$.				
			
			\item 
			\begin{enumerate}[\rm(a)]
				\item When $m=1$, $A$ is a UFD, $\dk(A)=A$ and $\left(\frac{ k_1[Z,T]}{(f(Z,T))}\right)^{*} = {k_1}^{*}.$ 
				\item When $m>1$, $A\otimes_k k_1$ is a UFD, $\dk(A)=A$ and $\left(\frac{ k_1[Z,T]}{(f(Z,T))}\right)^{*} = {k_1}^{*}$.
			\end{enumerate}
			
			\item 
			$A\otimes_k\overline{k}$ is a UFD, $f(Z,T)\notin k$ and there exists a non-trivial exponential map $\psi$ on $A$ such that $\x, a\in A^{\psi}$ with $w_1(a)>0$ ($w_1$ is defined as in \thref{lin2}).
		\end{enumerate}
	\end{thm}
	
	\begin{proof}
		Since $\alpha(\X)$ is ${\bf r}$-divisible in the system of coordinates $\{X_1-\lambda_1,\dots,X_m-\lambda_m\}$,
		for each $i\in \{1,\dots,m\}$, there exists $\alpha_i\in k_1^{[m-i+1]}$ such that
		$$\alpha(\X) = (X_1-\lambda_1)^{r_1}\alpha_1(X_1-\lambda_1,\ldots,X_m-\lambda_m),\, X_1-\lambda_1\nmid \alpha_1$$ $${\alpha_{i-1}(0,X_i-\lambda_i,\dots,X_m-\lambda_m)}={(X_i-\lambda_i)^{r_i}}\alpha_i(X_{i}-\lambda_i,\dots,X_m-\lambda_m),\, X_i-\lambda_i\nmid \alpha_i,\text{ for } 2\leqslant i\leqslant m.$$ 
		Therefore, without loss of generality, we can assume that
		$$
		A^{\prime}:=
		A\otimes_k k_1 \cong\dfrac{{k_1}[\X,Y,Z,T]}{(X_1^{r_1}\alpha_1(\X)Y- f(Z,T)-h_1(\X,Z,T))},
		$$ 
		where $h_1(\X,Z,T) := h(X_1+\lambda_1,\dots, X_m+\lambda_m, Z,T)$. 
		Note that $X_1^{r_1}\alpha_1$ is now ${\bf r}$-divisible in the system of coordinates $\{\X\}$.
		
		We are going to prove the above equivalence of statements in the following sequence:
		$$
		\begin{tikzcd}[column sep=small]
			{\rm(i)} \arrow[r, Rightarrow]\arrow[d, Rightarrow]\arrow[rrd, Rightarrow] & {\rm(ii)} \arrow[r, Rightarrow]& {\rm(iv)}\arrow[r, Rightarrow] & {\rm(vi)}\arrow[r,Rightarrow]& {\rm (vii)}\arrow[r, Leftrightarrow]& {\rm(viii)}\arrow[r, Rightarrow]&{\rm(ix)}\arrow[r, Rightarrow] &{\rm(v)}\arrow[r, Rightarrow]& {\rm(i)} \\
			{\rm(xiv)} &  &{\rm(iii)}\arrow[r, Rightarrow]& {\rm(x)}\arrow[r, Rightarrow] &{\rm(xi)}\arrow[r, Leftrightarrow]&{\rm(xii)}\arrow[r, Rightarrow] & {\rm(xiii)}\arrow[ru, Rightarrow]& {\rm(xiv)}\arrow[u,Rightarrow]
		\end{tikzcd}
		$$
		\smallskip
		\noindent
		$\rm (i) \Rightarrow (xiv)$: 
		Since $k[\X,Y,Z,T]= k[\X,H]^{[2]}$, $A= k[\x]^{[2]}=k[\x,h_1,h_2]$, for some $h_1, h_2\in A$. Therefore, $A\otimes_k\overline{k}= \overline{k}[\x]^{[2]}$ is a UFD. 
		Further since $A$ is a domain with $A^*=k^*$, it follows that $f(Z,T) \notin k$.
		Now $y\in A$ and $w_1(y)>0$. Since $w_1(x_i)\leqslant 0$, for all $i\in \{1,\dots,m\}$, either $w_1(h_1)>0$ or $w_1(h_2)>0$. Without loss of generality, let $w_1(h_1)>0$. 
		Now we consider the exponential map $\phi: A\rightarrow A[U]$ defined as follows 
		\begin{center}
			$\phi(h_2) = h_2+U,\, \phi(h_1) = h_1,\, \phi(x_i) = x_i,\, i\in \{1,\dots,m\}$.
		\end{center}
		Then one can easily check that $A^{\phi} = k[\x,h_1]$.
		
		\smallskip
		\noindent
		$\rm (vi) \Rightarrow (vii)$ and $\rm (x) \Rightarrow (xi)$: By \thref{line}, we have, $k_1[Z,T] = k_1[f]^{[1]}$. 
		Since  $k_1$ is a separable extension over $k$, by \thref{sepco}, we have $k[Z,T] = k[f]^{[1]}$, i.e., $\frac{k[Z,T]}{(f(Z,T))}= k^{[1]}$. 
		
		\smallskip
		\noindent
		{$\rm (vii) \Leftrightarrow (viii)$ and $\rm (xi) \Leftrightarrow (xii)$: Let  $\p:=(x_1-\lambda_1,\dots,x_m-\lambda_m)\overline{k}[\x]\cap E\in Max(E)$. Then $\alpha\in \p$ and $\frac{E_{\p}}{\p E_{\p}}\hookrightarrow k_1$ is a separable field extension of $k$. Hence from \thref{fib1} the assertions follow.}
		
		\smallskip
		\noindent
		$\rm (viii) \Leftrightarrow (vii)\Rightarrow (ix)$ and $\rm(xii)\Leftrightarrow(xi)\Rightarrow (xiii):$ 
		Since $f(Z,T)$ is a line in $k[Z,T]$, $\frac{k_1[Z,T]}{(f(Z,T))}= k_1^{[1]}$ and $f(Z,T)$ is irreducible in $\overline{k}[Z,T]$. Hence $\left(\frac{k_1[Z,T]}{(f(Z,T))}\right)^{*} = k_1^{*}$ and by \thref{UFDC}, we have $A\otimes_k L$ is a UFD for every algebraic extension $L$ of $k$. Thus the assertions follow.

		\smallskip
		\noindent
		$ \rm(xiii) \Rightarrow (v)$: 
			
			\noindent
			$\rm(a)$:
			Let $m=1$. 
			Since $A$ is a UFD and $\left(\frac{{k_1}[Z,T]}{(f(Z,T))}\right)^{*} ={k_1}^{*}$, by \thref {ufdg},  $f(Z,T)$ is irreducible in $k_1[Z,T]$ as $k(\lambda_1)=k_1$. 
			Now as $\dk(A)=A$, we have $\dk(A^{\prime})=A^{\prime}$.
			Thus, by applying \thref{lin2}(ii) for $A^{\prime}$, we have $f(Z,T) = a_0(Z) + a_1(Z)T$, for some $a_0,a_1\in k_1^{[1]}$. Therefore by \thref{linear}, $k_1[Z,T] = k_1[f(Z,T)]^{[1]}$ and hence by \thref{sepco}, $f$ is a coordinate in $k[Z,T]$ as $k_1$ is separable over $k$.
			
			\smallskip
			\noindent
			\rm(b):
			Henceforth $m>1$. 
			$ A^{\prime}$ is a UFD and $\left(\frac{{k_1}[Z,T]}{(f(Z,T))}\right)^{*} ={k_1}^{*}$. Therefore, by \thref{UFDline}, $f(Z,T)$ is irreducible in $k_1[Z,T]$.
			Now since $\dk(A) = A$, it follows that $\dk( A^{\prime}) =  A^{\prime}$. Therefore, by \thref{lin2}(i) there exists an integral domain $\tilde{A}$ such that
			\begin{equation*}
				\tilde{A} = \dfrac{k_1[\X,Y,Z,T]}{(X_1^{r_1}X_2^{r_2}\dots X_m^{r_m}Y - f(Z,T))},
			\end{equation*}
			with $ \dk(\tilde{A})= \tilde{A}$.
			By \thref{GA2}, there exist $l\in \{1,\dots,m\}$ and an integral domain $\tilde{A_l}$ such that 
			\begin{equation*}
				\tilde{A_l} \cong \dfrac{k_1(X_l)[X_1,\dots,X_{l-1},X_{l+1},\dots,X_m,Y,Z,T]}{(X_1^{r_1}X_2^{s_2}\dots X_m^{s_m}Y- f(Z,T))}
			\end{equation*}
			and $\dk(\tilde{A_l}) = \tilde{A_l}$. Therefore, by \thref{com2}(i), there exist $Z_1,T_1\in k_1(X_l)[Z,T]$ and $a_0,a_1\in k_1(X_l)^{[1]}$ such that $k_1(X_l)[Z,T] = k_1(X_l)[Z_1,T_1]$ and 
			$$ 
			f(Z,T) = a_0(Z_1) + a_1(Z_1)T_1. 
			$$
			Now  $\left(\frac{k_1[Z,T]}{(f(Z,T))}\right)^{*} = k_1^{*}$, therefore $\left(\frac{k_1(X_l)[Z,T]}{(f(Z,T))}\right)^{*} =\left(\frac{k_1[Z,T]}{(f(Z,T))}[X_l]\otimes_{k_1[X_l]}{k_1(X_l)}\right)^{*}= k_1(X_l)^{*}$,
			and since $f$ is irreducible in $k_1[Z,T]$, $f$ is irreducible in $k_1(X_l)[Z,T]$.
			Thus, by \thref{linear}, $k_1(X_l)[Z,T] = k_1(X_l)[f(Z,T)]^{[1]}$
			and hence by \thref{sepco}, $f$ is a coordinate of $k[Z,T]$.

			\noindent
			$\rm (ix) \Rightarrow (v)$:
			When $m=1$, since $A$ is a UFD and $\left(\frac{{k_1}[Z,T]}{f(Z,T)}\right)^{*} ={k_1}^{*},$ therefore by \thref {ufdg}, $f(Z,T)$ is irreducible in $k_1[Z,T]$.
			Now as $\ml(A)=k$, we have $\ml(A^{\prime})=k_1$.
			Thus, by applying \thref{lin}(ii) for $A^{\prime}$, we have $f(Z,T) = a_0(Z) + a_1(Z)T$, for some $a_0,a_1\in k_1^{[1]}$.
			
			Similarly, when $m>1,$
			$A^{\prime}$ is a UFD
			and $\left(\frac{{k_1}[Z,T]}{f(Z,T)}\right)^{*} ={k_1}^{*}$. Therefore, by \thref{UFDline}, $f(Z,T)$ is irreducible in $k_1[Z,T]$.
			Now since $\ml(A) = k$, it follows that $\ml(A^{\prime}) = k_1$. Therefore, by \thref{lin}(i), there exists an integral domain $\tilde{A}$ such that
			\begin{equation*}
				\tilde{A} := \dfrac{k_1[\X,Y,Z,T]}{(X_1^{r_1}X_2^{r_2}\dots X_m^{r_m}Y - f(Z,T))}
			\end{equation*}
			with $\dk(\tilde{A}) = \tilde{A}$. Following the same arguments as in $\rm(xiii)\Rightarrow \rm(v)$, the implication follows. 
			
			\noindent
			$\rm (xiv) \Rightarrow (v)$: We can extend $\psi$ to a non-trivial exponential map $\phi$ on $A^{\prime}$ (cf. \thref{lemma : properties}(iv)) such that $\x,a\in  (A^{\prime})^{\phi}$. Thus $a\in\dk(A^{\prime})$ and $w_1(a)>0$.
			The proof of \thref{lin2} shows that there exists an affine domain $\tilde{A}$ of the form 
			\begin{equation*}
				\tilde{A} =  \frac{k_1[\X,Y,Z,T]}{(X_1^{r_1}X_2^{r_2}\dots X_m^{r_m}Y-f(Z,T))} 
			\end{equation*}
			on which $\phi$ induces a non-trivial exponential map $\phi_1$. 
			For each $g \in A^{\prime}$, let  $\tilde{g}$ denote the image of $g$ in $\tilde{A}$.
			Again the proof of \thref{lin2} shows that $\tilde{y}\in \tilde{A}^{\phi_{1}}$ and $\tilde{g}\in \tilde{A}^{\phi_1}$, for all $g\in (A^{\prime})^{\phi}$.  
			Therefore, $\tilde{x}_1,\dots,\tilde{x}_m,\tilde{y}\in \tilde{A}^{\phi_1}$ and hence
			$k_1[\tilde{x}_1,\dots,\tilde{x}_m,\tilde{z},\tilde{t}]\subsetneqq\dk(\tilde{A})$ (cf. \thref{DK}).
			Now we consider two cases: 
			\\
			{\it{ Case 1:}}
			Let $k$ be an infinite field. 
			Without loss of generality, by \thref{com2}(i), we have $f(Z,T) = a_0(Z)+a_1(Z)T$ for some $a_0,a_1\in k_1^{[1]}$. 
			Now since $A\otimes_k\overline{k}$ is a UFD and $f\notin k$, by \thref{UFDline}, we have $f(Z,T)$ is irreducible in $\overline{k}[Z,T]$. 
			
			If $a_1(Z) =0$, then $f(Z,T)=a_0(Z)$ is irreducible in $\overline{k}[Z,T]$. Therefore, $f$ is linear in $Z$. Hence ${k[Z,T]}=k[{f(Z,T)}]^{[1]}.$ 
			
			Next let us assume that $a_1(Z)\neq 0$. 
			By \thref{lemma : properties}(iii), $\phi_1$ induces a non-trivial exponential map on 
			\begin{equation*}
				\tilde{A_1}:=\tilde{A}\otimes_{k_1[\tilde{x}_1,\dots,\tilde{x}_m,\tilde{y}]}k_1(\tilde{x}_1,\dots,\tilde{x}_m,\tilde{y})\cong\frac{k_1(\tilde{x}_1,\dots,\tilde{x}_m,\tilde{y})[Z,T]}{(\beta-a_0(Z)-a_1(Z)T)},
			\end{equation*}
			for some $\beta\in k_1(\tilde{x}_1,\dots,\tilde{x}_m,\tilde{y})$.
			Since $f$ is irreducible, $\gcd(a_0(Z),a_1(Z))=1$.
			Then it follows that $\gcd(\beta-a_0(Z),a_1(Z))=1$. Now as $\tilde{A_1}$ is not rigid, it follows that $a_1(Z) \in k^*$, and thus
			${k[Z,T]}=k[{f(Z,T)}]^{[1]}.$\\
			{\it{Case 2:}}
			Let $k$ be a finite field.
			Then, as in {\it Case 1}, we can conclude that ${\overline{k}[Z,T]}=\overline{k}[{f(Z,T)}]^{[1]}$. 
			Since $\overline{k}$ is a separable extension of $k$, by \thref{sepco}, we have ${k[Z,T]}=k[{f(Z,T)}]^{[1]}.$
			
			\noindent
			$\rm (v) \Rightarrow (i):$
			Follows from \thref{RS}.
			
			The rest of the equivalences follows trivially.		
		\end{proof}
		
		Note that the family of hypersurfaces given by
		$$
		a_1(X_1) \cdots a_m(X_m) Y -f(Z,T)-  h(X_1,\ldots, X_m,Z,T),
		$$
		where every prime divisor of $a_1(X_1) \cdots a_m(X_m)$ in $k[X_1,\ldots,X_m]$ divides $h$, and every $a_i(X_i)$ has a separable multiple root $\lambda_i$ over $k$, are included in the family of hypersurfaces mentioned in the above theorem.

		Now we state a result connecting \thref{main} and the Zariski Cancellation Problem.
		\begin{cor}\thlabel{czcp}
			Let $k$ be a field of positive characteristic.
			Suppose that $f(Z,T)$ is a non-trivial line over $k[Z,T]$ {\rm (\thref{NonL})}.
			Then an affine domain $A$ considered in \thref{main} 
			provides a counterexample to the ZCP and to the $\mathbb{A}^2$-fibration problem over $k^{[m]}$ in positive characteristic.
		\end{cor}
		\begin{proof}
			Since $f$ is a line, by \thref{p1}, $A^{[1]}= k^{[m+3]}$ and  by \thref{Fib2} $A$ is an $\A^2$-fibration over $k[\x]$.
			Now
			$k[Z,T]\neq k[f]^{[1]}$ therefore by \thref{main}((v)$\Leftrightarrow$(iv)), $A\neq k^{[m+2]}$. Hence $A$ is a counterexample to the ZCP.
			Again, by  \thref{main}((v)$\Leftrightarrow$(iii)), $A\not\cong k[\x]^{[2]}$ and 
			thus $A$ is a non-trivial 
			$\A^2$-fibration over $k[\x]$.
		\end{proof}

		
		\begin{rem}
			{\rm	Let $k$ be a field of characteristic zero, ${\bf r}=(r_1,\dots,r_m)\in \Z_{>0}^m$ and $H:= \alpha(\X)Y -f(Z,T)-h(\X,Z,T)$ be such that $f\neq 0$, $\alpha$ is ${\bf r}$-divisible in the system of coordinates $\{\X\}$ in $k^{[m]}$ and every prime factor of $\alpha$ divides $h$ in $k[\X,Z,T]$. Suppose that $k[\X,Y,Z,T]/(H)=k^{[m+2]}$. If $r_i=1$ for some $i,\mi$, then $\alpha$ has a simple prime factor in $k[\X]$ and hence by \thref{mainch0}(I), $k[\X,Y,Z,T]=k[\X,H]^{[2]}$.}
		\end{rem}

		\bigskip
		
		\noindent
		{\bf Acknowledgements.}
		The second author acknowledges Department of Science and Technology (DST), India for their INDO-RUSS project (DST/INT/RUS/RSF/P-48/2021).

	\end{document}